%% file: CCY.tex
\documentclass[12pt,reqno]{amsart}

\usepackage[left=3.5cm,right=3.5cm]{geometry}
\usepackage{amssymb}
\usepackage{graphicx}
\usepackage{amscd}
\usepackage[pagebackref]{hyperref}
\usepackage{color}
\usepackage{tabularx}
\usepackage[table]{xcolor}
\usepackage{float}
\usepackage{graphics,amsmath,amssymb}
\usepackage{amsthm}
\usepackage{amsfonts}
\usepackage{latexsym}
\usepackage{epsf}
\usepackage{xifthen}
\usepackage{mathrsfs}
\usepackage{dsfont}
\usepackage{makecell}
\usepackage{subfig}
\usepackage{amsmath}
\allowdisplaybreaks[4]
\usepackage{listings}
\usepackage{etoolbox}
\usepackage{fancyhdr}
\usepackage{pdflscape}
\usepackage[title,toc,titletoc]{appendix}
\usepackage{enumitem}
\usepackage[noadjust]{cite}
\usepackage{tikz}
\usetikzlibrary{automata,positioning,arrows}
\usepackage{young}
\usepackage[object=vectorian]{pgfornament} 
\usepackage{lipsum,tikz}
\usepackage{multirow}
\usepackage[OT2,T1]{fontenc}
\usepackage{mathtools}

\hypersetup{
	colorlinks=true, 
	linktoc=all, 
	linkcolor=blue} 

\numberwithin{equation}{section}

\theoremstyle{theorem}
\newtheorem{theorem}{Theorem}[section]
\newtheorem*{theorem*}{Theorem}

\newtheorem{corollary}[theorem]{Corollary}
\newtheorem{lemma}[theorem]{Lemma}
\newtheorem{proposition}[theorem]{Proposition}

\newtheorem{innercustomgeneric}{\customgenericname}
\providecommand{\customgenericname}{}
\newcommand{\newcustomtheorem}[2]{%
	\newenvironment{#1}[1]
	{%
		\renewcommand\customgenericname{#2}%
		\renewcommand\theinnercustomgeneric{##1}%
		\innercustomgeneric
	}
	{\endinnercustomgeneric}
}
\newcustomtheorem{ctheorem}{Theorem}
\newcustomtheorem{clemma}{Lemma}

\theoremstyle{definition}

\newtheorem*{example*}{Example}
\newtheorem*{examples*}{Examples}
\newtheorem{remark}[theorem]{Remark}
\newtheorem*{remark*}{Remark}
\newtheorem*{remarks*}{Remarks}
\newtheorem*{note*}{Note}

\newtheoremstyle{named}{}{}{\itshape}{}{\bfseries}{.}{.5em}{#1\thmnote{ #3}}
\theoremstyle{named}

\DeclareMathAlphabet{\mydutchcal}{U}{dutchcal}{m}{n}
\DeclareMathAlphabet{\mybbm}{U}{BOONDOX-ds}{m}{n}
\DeclareSymbolFont{cyrletters}{OT2}{wncyr}{m}{n}

\DeclareMathSymbol{\cyb}{\mathalpha}{cyrletters}{'142}
\DeclareMathSymbol{\cyB}{\mathalpha}{cyrletters}{'102}
\DeclareMathSymbol{\cyD}{\mathalpha}{cyrletters}{'104}
\DeclareMathSymbol{\cydje}{\mathalpha}{cyrletters}{'016}
\DeclareMathSymbol{\cyL}{\mathalpha}{cyrletters}{'114}
\newcommand{\hcyD}{\widehat{\cyD}}
\newcommand{\hcyL}{\widehat{\cyL}}
\newcommand{\hcydje}{\widehat{\cydje}}
\newcommand{\tcyb}{\widetilde{\cyb}}

\newcommand{\Supp}{\operatorname{Supp}}
\newcommand{\Ind}{\operatorname{Ind}}

\newcommand{\ii}{\mybbm{i}}

\newcommand{\RP}{\operatorname{RP}}

\newcommand{\hp}{\hat{p}}
\newcommand{\hT}{\hat{T}}
\newcommand{\hP}{\widehat{P}}

\newcommand{\BB}{\mydutchcal{B}}
\newcommand{\DD}{\mydutchcal{D}}
\newcommand{\hDD}{\widehat{\mydutchcal{D}}}
\newcommand{\LL}{\mydutchcal{L}}
\newcommand{\hLL}{\widehat{\mydutchcal{L}}}
\newcommand{\RR}{\mydutchcal{R}}

\newcommand{\frakD}{\mathfrak{D}}
\newcommand{\frakA}{\mathfrak{A}}
\newcommand{\hfrakA}{\widehat{\mathfrak{A}}}

\newcommand{\rH}{\mathrm{H}}

\title[KKS determinants]{Domino tilings, nonintersecting lattice paths and subclasses of Koutschan--Krattenthaler--Schlosser determinants}

\author[Q. Chen]{Qipin Chen}
\address[Q. Chen]{Amazon, Seattle, WA 98109, USA}
\email{qipinche@amazon.com}

\author[S. Chern]{Shane Chern}
\address[S. Chern]{Fakult\"at f\"ur Mathematik, Universit\"at Wien, Oskar-Morgenstern-Platz 1, Wien 1090, Austria}
\email{chenxiaohang92@gmail.com, xiaohangc92@univie.ac.at}

\author[A. Yoshida]{Atsuro Yoshida}
\address[A. Yoshida]{Fakult\"at f\"ur Mathematik, Universit\"at Wien, Oskar-Morgenstern-Platz 1, Wien 1090, Austria}
\email{atsuro.yoshida@univie.ac.at}

\date{}



\begin{document}
	
\sloppy

\begin{abstract}
	Koutschan, Krattenthaler and Schlosser recently considered a family of binomial determinants. In this work, we give combinatorial interpretations of two subclasses of these determinants in terms of domino tilings and nonintersecting lattice paths, thereby partially answering a question of theirs. Furthermore, the determinant evaluations established by Koutschan, Krattenthaler and Schlosser produce many product formulas for our weighted enumerations of domino tilings and nonintersecting lattice paths. However, there are still two enumerations left corresponding to conjectural formulas made by the three. We hereby prove the two conjectures using the principle of holonomic Ansatz plus the approach of modular reduction for creative telescoping, and hence fill the gap.
\end{abstract}

\maketitle

\section{Introduction}

\subsection{Background}

Koutschan, Krattenthaler and Schlosser~\cite{KKS2025} recently launched the study of a family of \emph{binomial} determinants (\emph{KKS determinants} for short):
\begin{align}\label{eq:KKS-det-1}
	\underset{{0\le i,j\le n-1}}{\det} \left(l^{i+b}\binom{mj+i+c}{mj+a} + \binom{mj-i+d}{mj+a}\right),
\end{align}
where $a,b,c,d,l,m\in \mathbb{Z}$ are given parameters. In particular, their research focuses on the determinants that factorize well, that is, those satisfying a \emph{product}-like evaluation for all $n\ge 1$.

The overture of the work of Koutschan, Krattenthaler and Schlosser revolves around a combinatorial context introduced by Di Francesco. In 2021, Di Francesco~\cite{DF2021} considered domino tilings of Aztec triangles as a continuation of the work of Guitter and himself~\cite{DFG2020} on the twenty-vertex model, an object of interest to physicists \cite{Bax1982, Kel1974}. What is referred to as an \emph{Aztec triangle} $\frakA(n)$ is the domain taking the form of the left of Fig.~\ref{fig:aztec}, while the right side presents an instance of a tiling of this domain by \emph{dominoes}. See Remark~\ref{rmk:Aztec} for a concrete definition of $\frakA(n)$.

\begin{figure}[ht]
	\begin{tikzpicture}[scale=.5]
		\input{tikz/aztec4.tex}
	\end{tikzpicture}\qquad\qquad\qquad
	\begin{tikzpicture}[scale=.5]
		\input{tikz/aztec4.tex}
		\foreach\i\j in {0/0, 0/2, 0/4, 0/6, 1/6, 2/6, 3/4, 3/6, 4/4, 4/6, 5/6, 6/6}
		{
			\draw[line width=2pt,draw=blue] (\i,\j) rectangle (\i+1,\j+2);
		}
		\foreach\i\j in {1/2, 1/3, 1/4, 1/5, 1/8, 2/9, 3/8, 3/10, 4/9, 5/8}
		{
			\draw[line width=2pt,draw=blue] (\i,\j) rectangle (\i+2,\j+1);
		}
	\end{tikzpicture}
	\bigskip
	\caption{The Aztec triangle $\mathfrak{A}(4)$ and its domino tiling}\label{fig:aztec}
\end{figure}
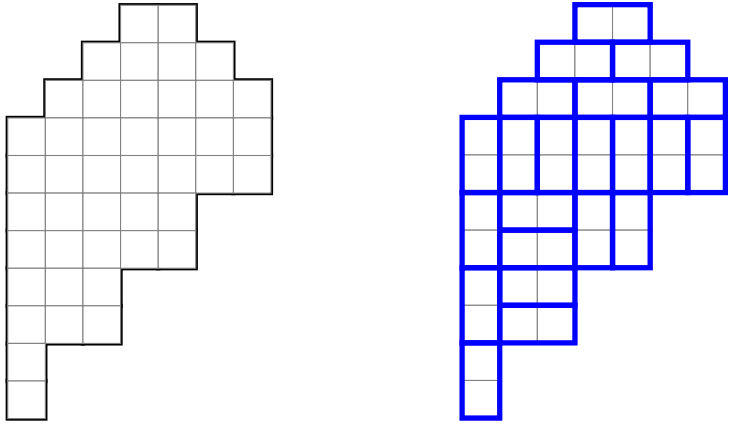

In \cite[Theorems~4.3 and 8.2]{DF2021}, Di Francesco formulated the enumeration of domino tilings of $\frakA(n)$, denoted by $\cyD(n)$, as a binomial determinant:
\begin{align}\label{eq:DF}
	\cyD(n) = \frac{1}{2}\cdot \underset{{0\le i,j\le n-1}}{\det} \left(2^{i}\binom{2j+i+1}{2j+1} - \binom{i-1}{2j+1}\right),
\end{align}
where we adopt the following convention for the \emph{binomial coefficients}:
\begin{align*}
	\binom{\alpha}{p}=
	\begin{cases}
		\frac{\alpha(\alpha-1)\cdots (\alpha-p+1)}{p!}, & \text{if $p\ge 0$},\\ 0, & \text{if $p<0$}.
	\end{cases}
\end{align*}
It is notable that the convention used in the original paper of Di Francesco~\cite{DF2021} was slightly different; he also put $\binom{\alpha}{p}=0$ for $-1\le \alpha< p$. From \eqref{eq:DF}, Di Francesco further conjectured a product formula \cite[p.~43, Conjecture~8.1]{DF2021} for $\cyD(n)$:
\begin{align}
	\cyD(n) = 2^{\frac{n(n-1)}{2}} \prod_{i=0}^{n-1} \frac{(4i+2)!}{(n+2i+1)!}.
\end{align}
This conjecture was presented by Di Francesco himself as a challenge at the 9th International Conference on \emph{Lattice Path Combinatorics and Applications}, and an affirmative answer by Christoph Koutschan was announced by Doron Zeilberger on the last day of the conference, June 25, 2021. Later on, Koutschan's proof was included as \cite[Sect.~4]{KKS2025}.

To connect Di Francesco's binomial determinant \eqref{eq:DF} with the KKS determinants given in \eqref{eq:KKS-det-1}, we first recall the work of Corteel, Huang and Krattenthaler~\cite{CHK2023} on domino tilings of generalized Aztec triangles, in which Di Francesco's setting is extended. In particular, it was shown in \cite[Theorem~4.1, eq.~(4.1)]{CHK2023} that for nonnegative integers $k$, the number $\cyD_{k}(n)$ of domino tilings of a certain generalized Aztec triangle (see Remark~\ref{rmk:CHK} for a concrete description) admits the product form
\begin{align}\label{eq:CHK}
	\cyD_{k}(n) = \prod_{j=1}^{k-1} (2j+1)^{j-k} \prod_{i\ge 0} \left(\prod_{s=-2k+4i+1}^{-k+2i} (2n+s) \prod_{s=k-2i}^{2k-4i-2}(2n+s)\right),
\end{align}
which reduces to Di Francesco's case of Aztec triangles by letting $k=n$. On the other hand, Koutschan, Krattenthaler and Schlosser showed in \cite[p.~12, Lemma~7]{KKS2025} that
\begin{align}\label{eq:KKS-cyD}
	\cyD_{n}(n+\delta) = \frac{1}{2}\cdot \underset{{0\le i,j\le n-1}}{\det} \left(2^{i}\binom{2j+i+2\delta+1}{2j+1} + \binom{2j-i+2\delta+1}{2j+1}\right),
\end{align}
which holds for nonnegative integers $\delta$. Although this determinant at $\delta=0$ looks seemingly different from Di Francesco's determinant in \eqref{eq:DF}, the equality between them is not a coincidence because $\binom{2j-i+1}{2j+1} = - \binom{i-1}{2j+1}$ for $i,j\ge 0$.
%

Now relating \eqref{eq:CHK} and \eqref{eq:KKS-cyD} suggests the following product formula for a subclass of KKS determinants for integers $\kappa$ \cite[p.~12, Theorem~6, eq.~(5.3)]{KKS2025}:
\begin{align}
	&\underset{{0\le i,j\le n-1}}{\det} \left(2^{i}\binom{2j+i+\kappa+1}{2j+1} + \binom{2j-i+\kappa+1}{2j+1}\right)\notag\\
	&\qquad\qquad = 2^{n(n-1)+1} \prod_{i=1}^n \frac{\Gamma(i)\Gamma(2i+\kappa)\Gamma(4i+\kappa-1)\Gamma(\frac{3i+\kappa-2}{2})}{\Gamma(2i)\Gamma(3i+\kappa)\Gamma(3i+\kappa-2)\Gamma(\frac{i+\kappa}{2})},
\end{align}
where $\Gamma(x)$ is the \emph{gamma function}. Therefore, the counting function $\cyD_{n}(n+\delta)$ for domino tilings forms a subclass of KKS determinants in a combinatorial context; this is especially adorable because of the melody played by the above \emph{product}-like evaluation.

Noting also that after certain modifications, some of their product formulas (\cite[p.~12, Theorem~6, eq.~(5.3)]{KKS2025} and \cite[p.~22, Theorem~15, eq.~(8.9)]{KKS2025}) appear coincidentally in the context of $(-1)$-enumerations of arrowed Gelfand--Tsetlin patterns introduced by Fischer and Schreier-Aigner (\cite[p.~3, Theorem~1]{FSA2024} and \cite[p.~4, Theorem~2]{FSA2024}), Koutschan, Krattenthaler and Schlosser proposed a request at the end of their paper \cite[p.~32]{KKS2025} for combinatorial interpretations of their determinants. Here we remark that the connection between KKS determinants and Fischer and Schreier-Aigner's arrowed Gelfand--Tsetlin patterns is still unclear.

The question of Koutschan, Krattenthaler and Schlosser sets up the \emph{overture} of our work.

\subsection{Main results}

From now on, our attention is restricted to two subclasses of KKS determinants with the number of free parameters reduced from six to three:
\begin{align}\label{eq:KKS-sub-D}
	\underset{{0\le i,j\le n-1}}{\det} \left(l^{i}\binom{mj+i+a}{mj+a} + \binom{mj-i+a}{mj+a}\right)
\end{align}
and
\begin{align}\label{eq:KKS-sub-H}
	\underset{{0\le i,j\le n-1}}{\det} \left(l^{i+1}\binom{mj+i+a+1}{mj+a+1} + \binom{mj-i+a-1}{mj+a+1}\right).
\end{align}
To facilitate our analysis, we introduce the determinant
\begin{align}\label{eq:det-B-def}
	\cyB_{a,b,c,d}^{(m,l)}(n) := \underset{{0\le i,j\le n-1}}{\det} \big(\cyb_{a,b,c,d}^{(m,l)}(i,j)\big),
\end{align}
where
\begin{align*}
	\cyb_{a,b,c,d}^{(m,l)}(i,j) := l^{j+b}\binom{mi+j+c}{mi+a} + \binom{mi-j+d}{mi+a}.
\end{align*}
It is clear that this determinant has the same value as the KKS determinant in \eqref{eq:KKS-det-1} because the associated matrices are transposes of each other.

For the combinatorial side inquired by Koutschan, Krattenthaler and Schlosser, we focus on two different objects --- \emph{domino tilings} and \emph{nonintersecting lattice paths} --- to be introduced in Sect.~\ref{sec:tilings-paths-def}, so as to interpret the subclasses of KKS determinants given in \eqref{eq:KKS-sub-D} and \eqref{eq:KKS-sub-H}. Undoubtedly, the two subclasses may also be explained by other combinatorial objects such as sequences of partitions and super symplectic tableaux. However, the bijections between them and domino tilings or nonintersecting lattice paths can be easily constructed, as shown in \cite{CHK2023}, and therefore we will not spill ink on them in this work.

Briefly speaking, we are going to count domino tilings of two types of domains generalized by Di Francesco's Aztec triangles \cite{DF2021}. In particular, such domains will be characterized by an arithmetic progression $(s(n-1)+r,s(n-2)+r,\ldots,r)$ for certain $s,r\in \mathbb{N}$. Furthermore, we assign three weights $w_1,w_2,w_3$ to different types of dominoes. Such weighted enumerations of domino tilings of the two types of domains are denoted by $\cyD_{w_1,w_2,w_3}^{(s,r)}(n)$ and $\hcyD_{w_1,w_2,w_3}^{(s,r)}(n)$, respectively; see \eqref{eq:cyD-def} and \eqref{eq:hcyD-def} for their concrete definitions.

Meanwhile, in light of a bijection due to Corteel, Huang and Krattenthaler~\cite{CHK2023}, we will also consider nonintersecting Delannoy and H-Delannoy paths, respectively, while these nonintersecting paths will be characterized by the same arithmetic progression with parameters $s,r\in \mathbb{N}$. Once again, we assign weights $w_1,w_2,w_3$ to three types of steps in the paths and conduct similar weighted enumerations by the counting functions $\cyL_{w_1,w_2,w_3}^{(s,r)}(n)$ and $\hcyL_{w_1,w_2,w_3}^{(s,r)}(n)$ defined in \eqref{eq:cyL-def} and \eqref{eq:hcyL-def}, respectively.

The main purpose of the present work is to offer combinatorial interpretations to the two subclasses \eqref{eq:KKS-sub-D} and \eqref{eq:KKS-sub-H} of KKS determinants by the aforementioned weighted enumerating functions. More precisely, we are going to establish the following two families of relations (see also Theorems~\ref{th:D-L} and \ref{th:L-B}):

\begin{theorem}\label{th:main-d}
	For $m\ge 1$ and $a\ge 0$,
	\begin{align}\label{eq:main-d}
		\cyD_{1,l-1,1}^{(m-1,a)}(n) = \cyL_{1,l-1,1}^{(m-1,a)}(n) = \tfrac{1}{2} \cyB_{a,0,a,a}^{(m,l)}(n),
	\end{align}
	where according to \eqref{eq:det-B-def},
	\begin{align*}
		\cyB_{a,0,a,a}^{(m,l)}(n) = \underset{{0\le i,j\le n-1}}{\det} \left(l^{j}\binom{mi+j+a}{mi+a} + \binom{mi-j+a}{mi+a}\right).
	\end{align*}
\end{theorem}

\begin{theorem}\label{th:main-h}
	For $m\ge 1$ and $a\ge 0$,
	\begin{align}\label{eq:main-h}
		\hcyD_{1,l-1,1}^{(m-1,a)}(n) = \hcyL_{1,l-1,1}^{(m-1,a)}(n) = \cyB_{a+1,1,a+1,a-1}^{(m,l)}(n),
	\end{align}
	where according to \eqref{eq:det-B-def},
	\begin{align*}
		\cyB_{a+1,1,a+1,a-1}^{(m,l)}(n) =  \underset{{0\le i,j\le n-1}}{\det} \left(l^{j+1}\binom{mi+j+a+1}{mi+a+1} + \binom{mi-j+a-1}{mi+a+1}\right).
	\end{align*}
\end{theorem}

It is fair to admit that our initial goal was to search for nice \emph{product formulas} for the weighted enumerations of domino tilings of Aztec-type domains. Fortunately, with the results of Koutschan, Krattenthaler and Schlosser in \cite{KKS2025}, the following identities are clear.

\begin{corollary}
	For all $n\ge 1$,
	\begin{align}
		\cyD_{1,1,1}^{(1,0)}(n) &= \tfrac{1}{2} \cyB_{0,0,0,0}^{(2,2)}(n) = 2^{n(n-1)} \prod_{i=1}^n \frac{\Gamma(i)\Gamma(2i)\Gamma(4i-3)\Gamma(\frac{3i-1}{2})}{\Gamma(2i-1)\Gamma(3i-1)\Gamma(3i-2)\Gamma(\frac{i+1}{2})},\label{eq:D-111-10}\\
		\cyD_{1,1,1}^{(1,1)}(n) &= \tfrac{1}{2} \cyB_{1,0,1,1}^{(2,2)}(n) = 2^{n(n-1)} \prod_{i=1}^n \frac{\Gamma(i)\Gamma(4i-1)\Gamma(\frac{3i-2}{2})}{\Gamma(3i)\Gamma(3i-2)\Gamma(\frac{i}{2})},\label{eq:D-111-11}\\
		\cyD_{1,3,1}^{(1,1)}(n) &= \tfrac{1}{2} \cyB_{1,0,1,1}^{(2,4)}(n) = 3^{\frac{1}{2}n(n-1)} \prod_{i=1}^{n} \frac{\Gamma(3i-1) \Gamma(\frac{i+1}{2})}{\Gamma(2i) \Gamma(\frac{3i-1}{2})},\label{eq:D-131-11}\\
		\cyD_{1,2,1}^{(2,0)}(n) &= \tfrac{1}{2} \cyB_{0,0,0,0}^{(3,3)}(n) = 2^{\frac{1}{2}n(n-1)} \prod_{i=1}^{n} \frac{\Gamma(4i-3) \Gamma(\frac{i+1}{3})}{\Gamma(3i-2) \Gamma(\frac{4i-2}{3})},\label{eq:D-121-20}\\
		\cyD_{1,2,1}^{(2,1)}(n) &= \tfrac{1}{2} \cyB_{1,0,1,1}^{(3,3)}(n) = 2^{\frac{1}{2}n(n-3)} 3^{-n}\prod_{i=1}^n \frac{\Gamma(4i-1) \Gamma(\frac{i}{3})}{\Gamma(3i-1) \Gamma(\frac{4i}{3})},\label{eq:D-121-21}\\
		\cyD_{1,2,1}^{(2,2)}(n) &= \tfrac{1}{2} \cyB_{2,0,2,2}^{(3,3)}(n) = 2^{\frac{1}{2}n(n-5)} \prod_{i=1}^n \frac{\Gamma(4i+1) \Gamma(\frac{i+2}{3})}{\Gamma(3i+1) \Gamma(\frac{4i+2}{3})}.\label{eq:D-121-22}
	\end{align}
	In addition,
	\begin{align}
		\hcyD_{1,1,1}^{(1,0)}(n) &= \cyB_{1,1,1,-1}^{(2,2)}(n) = 2^{n(n-1)} \prod_{i=1}^n \frac{\Gamma(i)\Gamma(4i-1)\Gamma(\frac{3i-2}{2})}{\Gamma(3i)\Gamma(3i-2)\Gamma(\frac{i}{2})},\label{eq:H-111-10}\\
		\hcyD_{1,1,1}^{(1,1)}(n) &= \cyB_{2,1,2,0}^{(2,2)}(n) = \prod_{i=1}^n \frac{\Gamma(4i)\Gamma(\frac{i+2}{2})}{\Gamma(3i)\Gamma(\frac{3i+2}{2})},\label{eq:H-111-11}\\
		\hcyD_{1,3,1}^{(1,0)}(n) &= \cyB_{1,1,1,-1}^{(2,4)}(n) = 3^{\frac{1}{2}n(n+1)} \prod_{i=1}^{n} \frac{\Gamma(3i-1) \Gamma(\frac{i+1}{2})}{\Gamma(2i) \Gamma(\frac{3i-1}{2})}.\label{eq:H-131-10}
	\end{align}
\end{corollary}

Here \eqref{eq:D-111-10} comes from \cite[eq.~(5.11) with $x=0$]{KKS2025}; \eqref{eq:D-111-11} comes from \cite[eq.~(5.3) with $x=0$]{KKS2025}; \eqref{eq:D-131-11} comes from \cite[eq.~(8.1)]{KKS2025}; \eqref{eq:D-121-20} comes from \cite[eqs.~(7.9)+(7.13)+(7.6)]{KKS2025}; \eqref{eq:D-121-21} comes from \cite[eq.~(7.7)]{KKS2025}; and \eqref{eq:D-121-22} comes from \cite[eq.~(7.8)]{KKS2025}. Meanwhile, \eqref{eq:H-111-10} comes from \cite[eq.~(6.6)]{KKS2025}; \eqref{eq:H-111-11} comes from \cite[eq.~(6.7)]{KKS2025}; and \eqref{eq:H-131-10} comes from \cite[eqs.~(8.6)+(8.3)]{KKS2025}.

It is notable that among the KKS determinants associated to our combinatorial context, there are still two \emph{conjectural} product formulas raised by Koutschan, Krattenthaler and Schlosser, as recorded in \cite[eq.~(10.8)]{KKS2025} and \cite[eq.~(10.9)]{KKS2025}, respectively. Therefore, the ultimate goal of this work is to confirm the two conjectural evaluations, which will be shown in Sect.~\ref{sec:KKS-conj}.

\begin{theorem}\label{th:det-conj}
	For all $n\ge 1$,
	\begin{align}
		\cyD_{1,1,1}^{(3,3)}(n) = \tfrac{1}{2} \cyB_{3,0,3,3}^{(4,2)}(n) = \prod_{i=1}^{n} \frac{\Gamma(6i-1) \Gamma(\frac{i+3}{4})}{\Gamma(5i) \Gamma(\frac{5i-1}{4})},
	\end{align}
	and
	\begin{align}
		\hcyD_{1,1,1}^{(3,1)}(n) = \cyB_{2,1,2,0}^{(4,2)}(n) = \frac{\Gamma(n+1)}{\Gamma(2n+1)} \prod_{i=1}^{n} \frac{\Gamma(6i-1) \Gamma(\frac{i+2}{4})}{\Gamma(5i-1) \Gamma(\frac{5i-2}{4})}.
	\end{align}
\end{theorem}

\section{Domino tilings and nonintersecting lattice paths}\label{sec:tilings-paths-def}

The first \emph{act} features a \emph{duet} performed by the two combinatorial objects of our interest --- domino tilings and nonintersecting lattice paths.

To start with, we recall that a \emph{(generalized) partition of length $n$} is a sequence $\lambda:=(\lambda_{1},\lambda_{2},\ldots, \lambda_{n})$ of $n$ weakly decreasing \emph{nonnegative} integers $\lambda_1\ge \lambda_2\ge \cdots \ge \lambda_{n}\ge 0$. Throughout, all partitions are such generalized ones, with \emph{zero} allowed as a part.

\subsection{Nonintersecting lattice paths}

We begin with the simpler objects, nonintersecting lattice paths. For $(i_1,j_1),(i_2,j_2)\in \mathbb{Z}^2$, we first define a \emph{Delannoy path} from $(i_1,j_1)$ to $(i_2,j_2)$ as a lattice path using only east ($\rightarrow$), north ($\uparrow$) or northeast ($\nearrow$) steps. For example, in Fig.~\ref{fig:D-path}, we show all five Delannoy paths from $(0,0)$ to $(1,2)$. Furthermore, we define an \emph{H-Delannoy path} from $(i_1,j_1)$ to $(i_2,j_2)$ as a Delannoy path that does \emph{not} end with an east step. Therefore, in Fig.~\ref{fig:D-path}, only the first path is \emph{not} H-Delannoy because its last step is an east one.

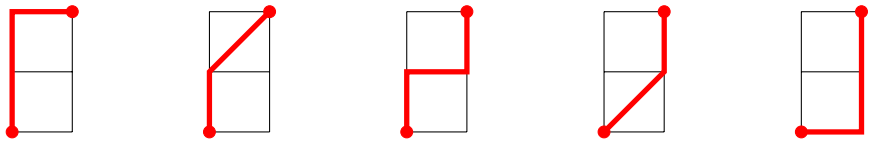
\begin{figure}[ht]
	\begin{tikzpicture}[scale=.8]
		\draw[black] (0,0) grid (1,2);
		\filldraw[red] (0,0) circle[radius=2.8pt];
		\filldraw[red] (1,2) circle[radius=2.8pt];
		\draw[line width=2pt,draw=red] (0,0) -- (0,1) -- (0,2) -- (1,2);
	\end{tikzpicture}\qquad\qquad
	\begin{tikzpicture}[scale=.8]
		\draw[black] (0,0) grid (1,2);
		\filldraw[red] (0,0) circle[radius=2.8pt];
		\filldraw[red] (1,2) circle[radius=2.8pt];
		\draw[line width=2pt,draw=red] (0,0) -- (0,1) -- (1,2);
	\end{tikzpicture}\qquad\qquad
	\begin{tikzpicture}[scale=.8]
		\draw[black] (0,0) grid (1,2);
		\filldraw[red] (0,0) circle[radius=2.8pt];
		\filldraw[red] (1,2) circle[radius=2.8pt];
		\draw[line width=2pt,draw=red] (0,0) -- (0,1) -- (1,1) -- (1,2);
	\end{tikzpicture}\qquad\qquad
	\begin{tikzpicture}[scale=.8]
		\draw[black] (0,0) grid (1,2);
		\filldraw[red] (0,0) circle[radius=2.8pt];
		\filldraw[red] (1,2) circle[radius=2.8pt];
		\draw[line width=2pt,draw=red] (0,0) -- (1,1) -- (1,2);
	\end{tikzpicture}\qquad\qquad
	\begin{tikzpicture}[scale=.8]
		\draw[black] (0,0) grid (1,2);
		\filldraw[red] (0,0) circle[radius=2.8pt];
		\filldraw[red] (1,2) circle[radius=2.8pt];
		\draw[line width=2pt,draw=red] (0,0) -- (1,0) -- (1,1) -- (1,2);
	\end{tikzpicture}
	\bigskip
	\caption{Delannoy paths from $(0,0)$ to $(1,2)$}\label{fig:D-path}
\end{figure}

Now let us fix a partition $\lambda=(\lambda_{1},\ldots, \lambda_{n})$ of length $n$.

A \emph{system of nonintersecting Delannoy paths marked by $\lambda$} is an $n$-tuple $p=(p_1,\ldots,p_n)$ of Delannoy paths such that $p_j$ moves from $(-j,j)$ to $(\lambda_j-j,n)$ for $1\le j\le n$ and that these Delannoy paths are pairwise \emph{nonintersecting} in the sense that every two different paths do not have any lattice point in common. We denote by $\LL_\lambda(n)$ the set of such nonintersecting Delannoy paths.

Similarly, a \emph{system of nonintersecting H-Delannoy paths marked by $\lambda$} is an $n$-tuple $\hp=(\hp_1,\ldots,\hp_n)$ of H-Delannoy paths such that $\hp_j$ moves from $(-j,j)$ to $(\lambda_j-j,n+1)$ for $1\le j\le n$ and that these H-Delannoy paths are pairwise nonintersecting. The set of nonintersecting H-Delannoy paths marked by $\lambda$ is written as $\hLL_\lambda(n)$.

\begin{figure}[ht]
	\begin{tikzpicture}[scale=.8]
		\draw[lightgray] (-3,1) grid (4,3);
		\filldraw[black] (-1,1) circle[radius=2.8pt];
		\filldraw[black] (4,3) circle[radius=2.8pt];
		\draw[line width=2pt,color=black] (-1,1) -- (3,1) -- (4,2) -- (4,3);
		\filldraw[black!40] (-2,2) circle[radius=2.8pt];
		\filldraw[black!40] (1,3) circle[radius=2.8pt];
		\draw[line width=2pt,color=black!40] (-2,2) -- (-1,3) -- (1,3);
		\filldraw[black!20] (-3,3) circle[radius=2.8pt];
		\filldraw[black!20] (-2,3) circle[radius=2.8pt];
		\draw[line width=2pt,color=black!20] (-3,3) -- (-2,3);
	\end{tikzpicture}\qquad\qquad
	\begin{tikzpicture}[scale=.8]
		\draw[lightgray] (-3,1) grid (4,4);
		\filldraw[black] (-1,1) circle[radius=2.8pt];
		\filldraw[black] (4,4) circle[radius=2.8pt];
		\draw plot[only marks,mark=x,mark size=4pt] (3,4);
		\draw[line width=2pt] (-1,1) -- (3,1) -- (4,2) -- (4,4);
		\filldraw[black!40] (-2,2) circle[radius=2.8pt];
		\filldraw[black!40] (1,4) circle[radius=2.8pt];
		\draw[color=black!40] plot[only marks,mark=x,mark size=4pt,color=black!40] (0,4);
		\draw[line width=2pt,color=black!40] (-2,2) -- (-1,3) -- (0,3) -- (1,4);
		\filldraw[black!20] (-3,3) circle[radius=2.8pt];
		\filldraw[black!20] (-2,4) circle[radius=2.8pt];
		\draw[color=black!20] plot[only marks,mark=x,mark size=4pt] (-3,4);
		\draw[line width=2pt,color=black!20] (-3,3) -- (-2,3) -- (-2,4);
	\end{tikzpicture}
	\bigskip
	\caption{Nonintersecting Delannoy (left) and H-Delannoy (right) paths marked by the partition $(5,3,1)$}\label{fig:non-p-ex}
\end{figure}
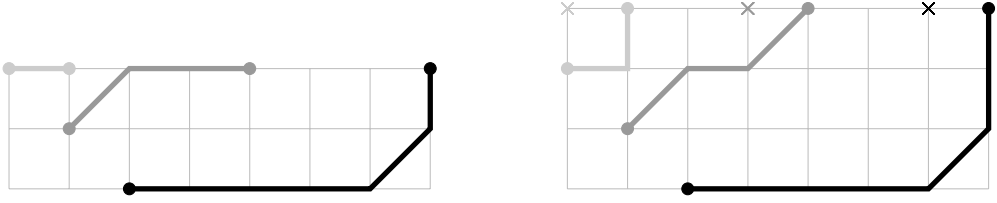

Let $w_1,w_2,w_3\in \mathbb{C}$ be weights. For every $p=(p_1,\ldots,p_n)\in \LL_\lambda(n)$ or $\hLL_\lambda(n)$, we assign a weight
\begin{align}
	w_{\mathrm{S}}(p) := w_1^{\#_p(\rightarrow)} w_2^{\#_p(\uparrow)} w_3^{\#_p(\nearrow)},
\end{align}
where we count each type of step over all $p_j$ of $p$. As an example, the nonintersecting Delannoy paths on the left of Fig.~\ref{fig:non-p-ex} have weight $w_1^7 w_2 w_3^2$, and the nonintersecting H-Delannoy paths on the right of Fig.~\ref{fig:non-p-ex} have weight $w_1^6 w_2^3 w_3^3$.

What lies at the center of our research are the following two \emph{weighted} enumerations:
\begin{align*}
	\cyL_{w_1,w_2,w_3}^{\lambda}(n) &:= \sum_{p\in \LL_{\lambda}(n)} w_{\mathrm{S}}(p),\\
	\hcyL_{w_1,w_2,w_3}^{\lambda}(n) &:= \sum_{\hp\in \hLL_{\lambda}(n)} w_{\mathrm{S}}(\hp).
\end{align*}
It is clear that for any $p\in \LL_{\lambda}(n)$,
\begin{align*}
	\#_p(\rightarrow) + \#_p(\nearrow) = \lambda_1 + \lambda_2 + \cdots + \lambda_n = |\lambda|,
\end{align*}
where $|\lambda|$ denotes the \emph{size} of the partition $\lambda$. In addition,
\begin{align*}
	\#_p(\uparrow) + \#_p(\nearrow) = (n-1) + (n-2) + \cdots + (n-n) = \tbinom{n}{2}.
\end{align*}
Likewise, for any $\hp\in \hLL_{\lambda}(n)$,
\begin{align*}
	\#_{\hp}(\rightarrow) + \#_{\hp}(\nearrow) = \lambda_1 + \lambda_2 + \cdots + \lambda_n = |\lambda|,
\end{align*}
and
\begin{align*}
	\#_{\hp}(\uparrow) + \#_{\hp}(\nearrow) = (n-0) + (n-1) + \cdots + (n-(n-1)) = \tbinom{n+1}{2}.
\end{align*}
The above observations imply the following relations for our weighted enumerations.
\begin{proposition}\label{eq:lattice-c-weight}
	Let $c\in \mathbb{C}$. Then for all $n\ge 1$,
	\begin{align*}
		\cyL_{cw_1,w_2,cw_3}^{\lambda}(n) &= c^{|\lambda|} \cyL_{w_1,w_2,w_3}^{\lambda}(n),\\
		\cyL_{w_1,cw_2,cw_3}^{\lambda}(n) &= c^{\binom{n}{2}} \cyL_{w_1,w_2,w_3}^{\lambda}(n).
	\end{align*}
	In addition,
	\begin{align*}
		\hcyL_{cw_1,w_2,cw_3}^{\lambda}(n) &= c^{|\lambda|} \hcyL_{w_1,w_2,w_3}^{\lambda}(n),\\
		\hcyL_{w_1,cw_2,cw_3}^{\lambda}(n) &= c^{\binom{n+1}{2}} \hcyL_{w_1,w_2,w_3}^{\lambda}(n).
	\end{align*}
\end{proposition}

For the moment, we concentrate on partitions that form an \emph{arithmetic progression}. Let $s,r\in \mathbb{N}$ be natural numbers. We fix our partition $\lambda$ as
\begin{align}\label{eq:lambda-sr}
	\lambda^{(s,r)} := (s(n-1)+r,s(n-2)+r,\ldots,r).
\end{align}

Let us denote by $\LL^{(s,r)}(n) := \LL_{\lambda^{(s,r)}}(n)$ the set of nonintersecting Delannoy paths marked by $\lambda^{(s,r)}$. We write
\begin{align}\label{eq:cyL-def}
	\cyL_{w_1,w_2,w_3}^{(s,r)}(n):=\sum_{p\in \LL^{(s,r)}(n)} w_{\mathrm{S}}(p).
\end{align}
Meanwhile, we denote by $\hLL^{(s,r)}(n) := \hLL_{\lambda^{(s,r)}}(n)$ the set of nonintersecting H-Delannoy paths marked by $\lambda^{(s,r)}$ and define
\begin{align}\label{eq:hcyL-def}
	\hcyL_{w_1,w_2,w_3}^{(s,r)}(n):=\sum_{\hp\in \hLL^{(s,r)}(n)} w_{\mathrm{S}}(\hp).
\end{align}

\begin{remark}\label{rmk:c1c2-scalar}
	Let $c_1,c_2\in \mathbb{C}\backslash\{0\}$. Then in view of Proposition~\ref{eq:lattice-c-weight}, we have
	\begin{align}
		\cyL_{c_1,(l-1)c_2,c_1c_2}^{(m-1,a)}(n) &= c_1^{(m-1)\binom{n}{2}+an} c_2^{\binom{n}{2}} \cyL_{1,l-1,1}^{(m-1,a)}(n),\label{eq:cyL-c1c2}\\
		\hcyL_{c_1,(l-1)c_2,c_1c_2}^{(m-1,a)}(n) &= c_1^{(m-1)\binom{n}{2}+an} c_2^{\binom{n}{2}+n} \hcyL_{1,l-1,1}^{(m-1,a)}(n).\label{eq:hcyL-c1c2}
	\end{align}
	So the KKS determinants in \eqref{eq:main-d} and \eqref{eq:main-h}, up to a scalar, may also be related to our enumerations of nonintersecting lattice paths counted by other choices of weights as indicated on the left of \eqref{eq:cyL-c1c2} and \eqref{eq:hcyL-c1c2}. A particularly interesting case here is to choose $c_1 = l-1$ and $c_2 = 1/(l-1)$ when $l\ne 1$, so that we arrive at expressions in terms of KKS determinants for $\cyL_{l-1,1,1}^{(m-1,a)}(n)$ and $\hcyL_{l-1,1,1}^{(m-1,a)}(n)$.
\end{remark}

\subsection{Domino tilings}

Now we move on to domino tilings of generalized Aztec triangles. Again, we fix a partition $\lambda=(\lambda_{1},\ldots, \lambda_{n})$ of length $n$. Our most crucial task is to construct the \emph{Aztec-type domains} for our tilings.

Let $M,N\in \mathbb{N}$ be natural numbers. We first introduce a domain $\frakA_{\mathrm{D}}(M;N)$ as a \emph{prototype} to facilitate our construction; this domain is built diagonal by diagonal. To start with, we work on a \emph{Cartesian grid} and label each unit square by its left-top point. Now for each $k$ with $0\le k\le N$, we place $M+\lceil\frac{k}{2}\rceil$ unit squares, labeled by $(i,i-k)$ with $0\le i\le M+\lceil\frac{k}{2}\rceil-1$, into our domain $\frakA_{\mathrm{D}}(M;N)$; these unit squares comprise the \emph{$k$-th diagonal} of the domain. We also color \emph{even-indexed} diagonals with \emph{white} and \emph{odd-indexed} diagonals with \emph{gray} in the way of a chessboard. See Fig.~\ref{fig:prototype-domain} for an example.

\begin{figure}[ht]
	\begin{tikzpicture}[scale=.5]
		\input{tikz/protodomain.tex}
	\end{tikzpicture}
	\bigskip
	\caption{The prototypical domain $\mathfrak{A}_D(4;8)$}\label{fig:prototype-domain}
\end{figure}
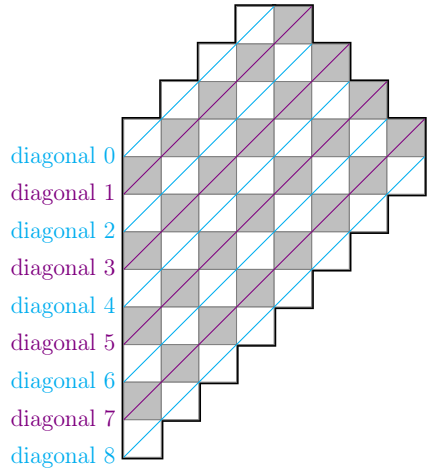

Next, we plot the \emph{Young diagram} of the partition $\lambda$. Note that the shape of this diagram can be uniquely characterized by its \emph{southeast} boundary, which consists of \emph{east} and \emph{north} edges. We can further encode this boundary by a sequence of \emph{circles} ($\circ$) and \emph{bullets} ($\bullet$), with each \emph{circle} denoting an \emph{east} edge and each \emph{bullet} denoting a \emph{north} edge; such an encoding is called a \emph{boundary encoding} of $\lambda$. We refer the reader to Fig.~\ref{fig:young} for an illustration of the following two boundary encodings:
\begin{align*}
	\begin{array}{rcl}
		(7,5,3,1) & \xrightarrow[\text{encoding}]{\text{boundary}} & \text{$\circ$ $\bullet$ $\circ$ $\circ$ $\bullet$ $\circ$ $\circ$ $\bullet$ $\circ$ $\circ$ $\bullet$}\\
		(6,4,2,0) & \xrightarrow[\text{encoding}]{\text{boundary}} & \text{$\bullet$ $\circ$ $\circ$ $\bullet$ $\circ$ $\circ$ $\bullet$ $\circ$ $\circ$ $\bullet$}
	\end{array}
\end{align*}

\begin{figure}[ht]
	\begin{tikzpicture}[scale=.5]
		\draw (0,0) grid (7,-1);
		\draw (0,-1) grid (5,-2);
		\draw (0,-2) grid (3,-3);
		\draw (0,-3) grid (1,-4);
		\draw[ultra thick] (0,-4) -- (1,-4) -- (1,-3) -- (3,-3) -- (3,-2) -- (5,-2) -- (5,-1) -- (7,-1) -- (7,0);
		\foreach\i\j in {1/-4, 3/-3, 5/-2, 7/-1} {
			\draw[black,fill=black] (\i+.25,\j+.5) circle (.1);
		}
		\foreach\i\j in {0/-4, 1/-3, 2/-3, 3/-2, 4/-2, 5/-1, 6/-1} {
			\draw[black,fill=white] (\i+.5,\j-.25) circle (.1);
		}
	\end{tikzpicture}\qquad\qquad
	\begin{tikzpicture}[scale=.5]
		\draw[white,fill=white] (0.5,-4.25) circle (.1);
		\draw (0,0) grid (6,-1);
		\draw (0,-1) grid (4,-2);
		\draw (0,-2) grid (2,-3);
		\draw[ultra thick] (0,-4) -- (0,-3) -- (2,-3) -- (2,-2) -- (4,-2) -- (4,-1) -- (6,-1) -- (6,0);
		\foreach\i\j in {0/-4, 2/-3, 4/-2, 6/-1} {
			\draw[black,fill=black] (\i+.25,\j+.5) circle (.1);
		}
		\foreach\i\j in {0/-3, 1/-3, 2/-2, 3/-2, 4/-1, 5/-1} {
			\draw[black,fill=white] (\i+.5,\j-.25) circle (.1);
		}
	\end{tikzpicture}
	\bigskip
	\caption{The boundary encoding of partitions $(7,5,3,1)$ (left) and $(6,4,2,0)$ (right)}\label{fig:young}
\end{figure}
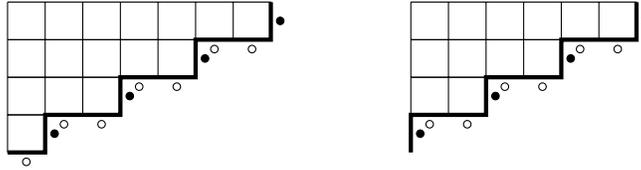

Recall that the partition $\lambda=(\lambda_{1},\ldots, \lambda_{n})$ is fixed. Now we introduce \emph{two} families of \emph{Aztec-type domains}, both of which are constructed by removing a number of unit squares from the last diagonal of a certain prototypical domain:

\begin{description}[itemsep=2pt]
	
	\item[Type 1] We start with the prototypical domain $\frakA_{\mathrm{D}}(\lambda_1;2n-1)$. Note that in this domain, there are $\lambda_1+n$ unit squares in the last diagonal, and this number equals the total number of circles and bullets in the boundary encoding of $\lambda$. Now in the last diagonal, we place these circles and bullets derived from the boundary encoding of $\lambda$ into the unit squares, arranging from left-bottom to right-top. We then remove the unit squares marked by a \textbf{bullet} from $\frakA_{\mathrm{D}}(\lambda_1;2n-1)$ to obtain the desired Aztec-type domain $\frakA_{\lambda}(n)$, which is called the \emph{Aztec-type domain of Type 1 marked by $\lambda$};
	
	\item[Type 2] We start with the prototypical domain $\frakA_{\mathrm{D}}(\lambda_1;2n)$. Note that in this domain, there are also $\lambda_1+n$ unit squares in the last diagonal. Now in the last diagonal, we again place these circles and bullets derived from the boundary encoding of $\lambda$ into the unit squares, arranging from left-bottom to right-top. We then remove the unit squares marked by a \textbf{circle} from $\frakA_{\mathrm{D}}(\lambda_1;2n)$ to obtain the desired Aztec-type domain $\hfrakA_{\lambda}(n)$, which is called the \emph{Aztec-type domain of Type 2 marked by $\lambda$}.
\end{description}

\begin{figure}[ht]
	\begin{tikzpicture}[scale=.4]
		\input{tikz/aztecdomain1.tex}
	\end{tikzpicture}\qquad\qquad
	\begin{tikzpicture}[scale=.4]
		\input{tikz/aztecdomain2.tex}
	\end{tikzpicture}
	\bigskip
	\caption{The Aztec-type domains $\mathfrak{A}_{\lambda}(4)$ (left) and $\widehat{\mathfrak{A}}_{\lambda}(4)$ (right) marked by the partition $\lambda=(7,5,3,1)$}\label{fig:aztec-domain-lambda}
\end{figure}
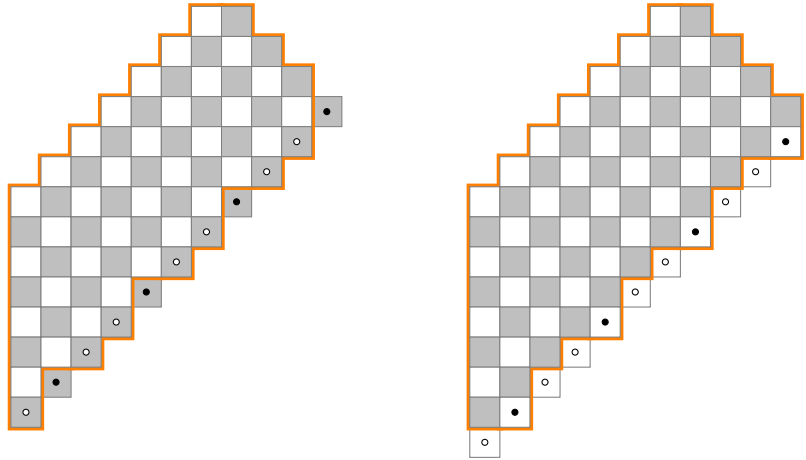

In Fig.~\ref{fig:aztec-domain-lambda}, the left side shows the Aztec-type domain of Type 1 marked by the partition $(7,5,3,1)$ and the right side shows the Aztec-type domain of Type 2 marked by the same partition.

\begin{remark}\label{rmk:Aztec}
	The original Aztec triangle $\frakA(n)$ introduced by Di Francesco is the Aztec-type domain of Type 1 marked by the partition $(n,n-1,\ldots,1)$.
\end{remark}

\begin{remark}\label{rmk:CHK}
	The domain involved in the counting function $\cyD_{k}(n)$ considered by Corteel, Huang and Krattenthaler is the Aztec-type domain of Type 1 marked by the partition $(k,k-1,\ldots,1,\underbrace{0, 0, \ldots, 0}_{n-k\text{ copies}})$.
\end{remark}

Let $\DD_\lambda(n)$ denote the set of domino tilings of the Aztec-type domain $\frakA_{\lambda}(n)$ and let $\hDD_\lambda(n)$ denote the set of domino tilings of the Aztec-type domain $\hfrakA_{\lambda}(n)$.

Note that our coloring of the Aztec-type domain introduces \emph{four} different dominoes:
\begin{align*}
	\begin{tikzpicture}[scale=1]
		\node[left] at (0, 0) {$\mathfrak{D}_1=$};
		\node[right] at (0.5, 0) {,};
		\draw[fill=lightgray,draw=gray,scale=1] (0,-0.5) rectangle (0.5,0);
		\draw[fill=white,draw=gray,scale=1] (0,0) rectangle (0.5,0.5);
		\draw[ultra thick,draw=blue] (0,-0.5) rectangle (0.5,0.5);
		\node[left] at (3, 0) {$\mathfrak{D}_2=$};
		\node[right] at (3.5, 0) {,};
		\draw[fill=white,draw=gray,scale=1] (3,-0.5) rectangle (3.5,0);
		\draw[fill=lightgray,draw=gray,scale=1] (3,0) rectangle (3.5,0.5);
		\draw[ultra thick,draw=blue] (3,-0.5) rectangle (3.5,0.5);
		\node[left] at (6, 0) {$\mathfrak{D}_3=$};
		\node[right] at (7, 0) {,};
		\draw[fill=lightgray,draw=gray,scale=1] (6,-0.25) rectangle (6.5,0.25);
		\draw[fill=white,draw=gray,scale=1] (6.5,-0.25) rectangle (7,0.25);
		\draw[ultra thick,draw=blue] (6,-0.25) rectangle (7,0.25);
		\node[left] at (9.5, 0) {$\mathfrak{D}_4=$};
		\node[right] at (10.5, 0) {.};
		\draw[fill=white,draw=gray,scale=1] (9.5,-0.25) rectangle (10,0.25);
		\draw[fill=lightgray,draw=gray,scale=1] (10,-0.25) rectangle (10.5,0.25);
		\draw[ultra thick,draw=blue] (9.5,-0.25) rectangle (10.5,0.25);
	\end{tikzpicture}
\end{align*}

Again, we let $w_1,w_2,w_3\in \mathbb{C}$ be weights. For every domino titling $T\in \DD_\lambda(n)$ or $\hDD_\lambda(n)$, we assign a weight
\begin{align}
	w_{\mathrm{D}}(T) := w_1^{\#_T(\frakD_1)} w_2^{\#_T(\frakD_2)} w_3^{\#_T(\frakD_3)} 1^{\#_T(\frakD_4)} = w_1^{\#_T(\frakD_1)} w_2^{\#_T(\frakD_2)} w_3^{\#_T(\frakD_3)},
\end{align}
where $\#_T(\frakD_i)$ counts the number of dominoes $\frakD_i$ in the tiling $T$ for each $i=1,2,3,4$.

Now the \emph{weighted} enumerations of our interest are
\begin{align*}
	\cyD_{w_1,w_2,w_3}^{\lambda}(n) &:= \sum_{T\in \DD_{\lambda}(n)} w_{\mathrm{D}}(T),\\
	\hcyD_{w_1,w_2,w_3}^{\lambda}(n) &:= \sum_{\hT\in \hDD_{\lambda}(n)} w_{\mathrm{D}}(\hT).
\end{align*}

We also require that the partitions in our study form an arithmetic progression as in \eqref{eq:lambda-sr}:
\begin{align*}
	\lambda^{(s,r)} := (s(n-1)+r,s(n-2)+r,\ldots,r),
\end{align*}
where $s,r\in \mathbb{N}$ are natural numbers.

Let $\DD^{(s,r)}(n) := \DD_{\lambda^{(s,r)}}(n)$ and
\begin{align}\label{eq:cyD-def}
	\cyD_{w_1,w_2,w_3}^{(s,r)}(n):=\sum_{T\in \DD^{(s,r)}(n)} w_{\mathrm{D}}(T).
\end{align}
Similarly, we write $\hDD^{(s,r)}(n) := \hDD_{\lambda^{(s,r)}}(n)$ and
\begin{align}\label{eq:hcyD-def}
	\hcyD_{w_1,w_2,w_3}^{(s,r)}(n):=\sum_{\hT\in \hDD^{(s,r)}(n)} w_{\mathrm{D}}(\hT).
\end{align}

\subsection{Bijections}

To close this \emph{act}, we will witness a \emph{piano four hands} with the rhythm centralized at bijections connecting our domino tilings and nonintersecting lattice paths.

\begin{lemma}\label{le:bij-1}
	For any $s,r\in \mathbb{N}$, there is a bijection
	\begin{align*}
		\begin{array}{cccc}
			\phi: & \DD^{(s,r)}(n) & \to & \LL^{(s,r)}(n)\\
			& T & \mapsto & p
		\end{array}
	\end{align*}
	such that
	\begin{align*}
		\#_T(\frakD_1) = \#_p(\rightarrow),\qquad \#_T(\frakD_2) = \#_p(\uparrow),\qquad \#_T(\frakD_3) = \#_p(\nearrow).
	\end{align*}
\end{lemma}

This bijection is presented in \cite[Sect.~3]{CHK2023} and here we visually display the idea. As a beginning, in the first three types of dominoes, we add an extra edge to produce the required lattice path. To be specific, we have
\begin{align*}
	\begin{tikzpicture}[scale=1]
		\node[left] at (0, 0) {$\mathfrak{D}_1=$};
		\node[right] at (0.5, 0) {,};
		\draw[fill=lightgray,draw=gray,scale=1] (0,-0.5) rectangle (0.5,0);
		\draw[fill=white,draw=gray,scale=1] (0,0) rectangle (0.5,0.5);
		\draw[ultra thick,draw=blue] (0,-0.5) rectangle (0.5,0.5);
		\draw[ultra thick,draw=red] (0,-0.25) -- (0.5,0.25);
		\node[left] at (3, 0) {$\mathfrak{D}_2=$};
		\node[right] at (3.5, 0) {,};
		\draw[fill=white,draw=gray,scale=1] (3,-0.5) rectangle (3.5,0);
		\draw[fill=lightgray,draw=gray,scale=1] (3,0) rectangle (3.5,0.5);
		\draw[ultra thick,draw=blue] (3,-0.5) rectangle (3.5,0.5);
		\draw[ultra thick,draw=red] (3,0.25) -- (3.5,-0.25);
		\node[left] at (6, 0) {$\mathfrak{D}_3=$};
		\node[right] at (7, 0) {,};
		\draw[fill=lightgray,draw=gray,scale=1] (6,-0.25) rectangle (6.5,0.25);
		\draw[fill=white,draw=gray,scale=1] (6.5,-0.25) rectangle (7,0.25);
		\draw[ultra thick,draw=blue] (6,-0.25) rectangle (7,0.25);
		\draw[ultra thick,draw=red] (6,0) -- (7,0);
		\node[left] at (9.5, 0) {$\mathfrak{D}_4=$};
		\node[right] at (10.5, 0) {.};
		\draw[fill=white,draw=gray,scale=1] (9.5,-0.25) rectangle (10,0.25);
		\draw[fill=lightgray,draw=gray,scale=1] (10,-0.25) rectangle (10.5,0.25);
		\draw[ultra thick,draw=blue] (9.5,-0.25) rectangle (10.5,0.25);
	\end{tikzpicture}
\end{align*}
Now given a domino tiling $T$ in $\DD^{(s,r)}(n)$, as shown on the left of Fig.~\ref{fig:bij-1}, we first flip it vertically to get the plot on the right. Then we make a rotation of $135$ degrees clockwise as shown at the bottom. The desired nonintersecting lattice paths $p$ are produced by these newly inserted edges in the dominoes.

\begin{figure}[ht]
	\begin{tikzpicture}[scale=.33]
		\input{tikz/bij1.tex}
	\end{tikzpicture}\qquad\qquad
	\begin{tikzpicture}[scale=.33,xscale=-1,yscale=1]
		\input{tikz/bij1.tex}
	\end{tikzpicture}
	
	\bigskip\bigskip
	\begin{tikzpicture}[scale=.33,xscale=-1,yscale=1,rotate=135,transform shape]
	\input{tikz/bij1.tex}
	\end{tikzpicture}
	\bigskip
	\caption{The bijection $\phi$ in Lemma~\ref{le:bij-1} for a tiling in $\DD^{(2,1)}(4)$}\label{fig:bij-1}
\end{figure}

Likewise, we have a bijection between $\hDD^{(s,r)}(n)$ and $\hLL^{(s,r)}(n)$ in the same manner, which is visualized in Fig.~\ref{fig:bij-2}.

\begin{lemma}\label{le:bij-2}
	For any $s,r\in \mathbb{N}$, there is a bijection
	\begin{align*}
		\begin{array}{cccc}
			\hat{\phi}: & \hDD^{(s,r)}(n) & \to & \hLL^{(s,r)}(n)\\
			& \hT & \mapsto & \hp
		\end{array}
	\end{align*}
	such that
	\begin{align*}
		\#_{\hT}(\frakD_1) = \#_{\hp}(\rightarrow),\qquad \#_{\hT}(\frakD_2) = \#_{\hp}(\uparrow),\qquad \#_{\hT}(\frakD_3) = \#_{\hp}(\nearrow).
	\end{align*}
\end{lemma}

\begin{figure}[ht]
	\begin{tikzpicture}[scale=.33]
		\input{tikz/bij2.tex}
	\end{tikzpicture}\qquad\qquad
	\begin{tikzpicture}[scale=.33,xscale=-1,yscale=1]
	\input{tikz/bij2.tex}
	\end{tikzpicture}
	
	\bigskip\bigskip
	\begin{tikzpicture}[scale=.33,xscale=-1,yscale=1,rotate=135,transform shape]
	\input{tikz/bij2.tex}
	\end{tikzpicture}
	\bigskip
	\caption{The bijection $\hat{\phi}$ in Lemma~\ref{le:bij-2} for a tiling in $\hDD^{(2,1)}(4)$}\label{fig:bij-2}
\end{figure}

In view of Lemmas~\ref{le:bij-1} and \ref{le:bij-2}, we immediately obtain the following relations that lead us to the first part in Theorems~\ref{th:main-d} and \ref{th:main-h}.

\begin{theorem}[$\Rightarrow$ First Part in Theorems~\ref{th:main-d} and \ref{th:main-h}]\label{th:D-L}
	For any $s,r\in \mathbb{N}$,
	\begin{align}
		\cyD_{w_1,w_2,w_3}^{(s,r)}(n) &= \cyL_{w_1,w_2,w_3}^{(s,r)}(n),\\
		\hcyD_{w_1,w_2,w_3}^{(s,r)}(n) &= \hcyL_{w_1,w_2,w_3}^{(s,r)}(n).
	\end{align}
\end{theorem}

\section{Delannoy and H-Delannoy determinants}

Following the introduction of our domino tilings and nonintersecting lattice paths, now let us perform a short \emph{recitative} to express the weighted enumerations $\cyL_{w_1,w_2,w_3}^{(s,r)}(n)$ and $\hcyL_{w_1,w_2,w_3}^{(s,r)}(n)$ for nonintersecting lattice paths as determinants involving weighted Delannoy and H-Delannoy numbers.

Let $w_1,w_2,w_3\in \mathbb{C}$ be weights. To begin with, we define the \emph{weighted Delannoy numbers} by
\begin{align}
	D_{w_1,w_2,w_3}(i,j) := \sum_{p_0: (0,0)\overset{D}{\to} (i,j)} w_1^{\#_{p_0}(\rightarrow)} w_2^{\#_{p_0}(\uparrow)} w_3^{\#_{p_0}(\nearrow)},
\end{align}
where the sum runs over all Delannoy paths $p_0$ from $(0,0)$ to $(i,j)$. In particular, they specialize to the usual \emph{Delannoy numbers} by setting $(w_1,w_2,w_3) = (1,1,1)$. In the same way, we define the \emph{weighted H-Delannoy numbers}, which reduce to the usual \emph{H-Delannoy numbers} after the same specialization, by
\begin{align}
	H_{w_1,w_2,w_3}(i,j) := \sum_{\hp_0: (0,0)\overset{H}{\to} (i,j+1)} w_1^{\#_{\hp_0}(\rightarrow)} w_2^{\#_{\hp_0}(\uparrow)} w_3^{\#_{\hp_0}(\nearrow)},
\end{align}
where the sum runs over all H-Delannoy paths $\hp_0$ from $(0,0)$ to $(i,j+1)$.

It is clear that
\begin{align}\label{eq:D-gf}
	\sum_{i,j\ge 0} D_{w_1,w_2,w_3}(i,j) u^i v^j = \frac{1}{1-w_1u-w_2v-w_3uv}.
\end{align}
In addition,
\begin{align*}
	H_{w_1,w_2,w_3}(i,j) = D_{w_1,w_2,w_3}(i,j+1) - w_1 D_{w_1,w_2,w_3}(i-1,j+1),
\end{align*}
which implies that
\begin{align}\label{eq:H-gf}
	\sum_{i,j\ge 0} H_{w_1,w_2,w_3}(i,j) u^i v^j &= (1-w_1 u) v^{-1} \left(\frac{1}{1-w_1u-w_2v-w_3uv} - \frac{1}{1-w_1u}\right)\notag\\
	&= \frac{w_2+w_3u}{1-w_1u-w_2v-w_3uv}.
\end{align}

Now we have the following relations.

\begin{lemma}
	For any $s,r\in \mathbb{N}$,
	\begin{align}
		\cyL_{w_1,w_2,w_3}^{(s,r)}(n) &= \underset{{0\le i,j\le n-1}}{\det} \big(D_{w_1,w_2,w_3}((s+1)i-j+r,j)\big),\label{eq:det-D}\\
		\hcyL_{w_1,w_2,w_3}^{(s,r)}(n) &= \underset{{0\le i,j\le n-1}}{\det} \big(H_{w_1,w_2,w_3}((s+1)i-j+r,j)\big).\label{eq:det-H}
	\end{align}
\end{lemma}

\begin{proof}
	It follows from the \emph{Lindstr\"om--Gessel--Viennot lemma} \cite{GV1985, Lin1973} that
	\begin{align*}
		\cyL_{w_1,w_2,w_3}^{(s,r)}(n) &= \underset{{1\le i,j\le n}}{\det} \big(D_{w_1,w_2,w_3}(s(n-i)-i+j+r,n-j)\big)\\
		&= \underset{{0\le i,j\le n-1}}{\det} \big(D_{w_1,w_2,w_3}(s(n-i-1)-i+j+r,n-j-1)\big).
	\end{align*}
	Now in the associated matrix of the latter determinant, we interchange the rows and columns by $(i,j)\mapsto (n-1-i,n-1-j)$, and notice that the determinant stays invariant. Hence, we arrive at \eqref{eq:det-D}. For \eqref{eq:det-H}, the same analysis may be applied.
\end{proof}

\section{Subclasses of KKS determinants}

It remains to connect the Delannoy and H-Delannoy determinants with the desired KKS determinants, and this is the theme of our second \emph{act}.

Let
\begin{align*}
	F(z_1,\ldots,z_k):=\sum_{n_1,\ldots,n_k=-\infty}^{\infty} f_{n_1,\ldots,n_k} z_1^{n_1}\cdots z_k^{n_k}
\end{align*}
be a \emph{Laurent series} in variables $z_1,\ldots,z_k$. We adopt the conventional notation that
\begin{align*}
	[z_1^{n_1}\cdots z_k^{n_k}]F(z_1,\ldots,z_k) := f_{n_1,\ldots,n_k}.
\end{align*}
Furthermore, by separating the variables $z_1,\ldots,z_k$ into two \emph{disjoint} groups $z_{i_1},\ldots,z_{i_l}$ and $z_{i_{l+1}},\ldots,z_{i_{k}}$, the \emph{regular part} of $F$ with respect to the variables $z_{i_1},\ldots,z_{i_l}$ is defined by
\begin{align*}
	\RP_{z_{i_1},\ldots,z_{i_l}} F(z_1,\ldots,z_k) := \sum_{n_{i_1},\ldots,n_{i_l}= 0}^{\infty} \sum_{n_{i_{l+1}},\ldots,n_{i_{k}}=-\infty}^{\infty} f_{n_1,\ldots,n_k} z_1^{n_1}\cdots z_k^{n_k},
\end{align*}
where $n_{i_j}$ represents the exponent of the variable $z_{i_j}$.

Recalling \eqref{eq:D-gf} and \eqref{eq:H-gf}, for the moment we write
\begin{align}
	P_{w_1,w_2,w_3}(u,v) &:= \frac{1}{1-w_1u-w_2v-w_3uv} = \sum_{i,j\ge 0} D_{w_1,w_2,w_3}(i,j) u^i v^j,\\
	\hP_{w_1,w_2,w_3}(u,v) &:= \frac{w_2+w_3u}{1-w_1u-w_2v-w_3uv} = \sum_{i,j\ge 0} H_{w_1,w_2,w_3}(i,j) u^i v^j.
\end{align}
Let us define
\begin{align*}
	\cydje_{w_1,w_2,w_3}^{(s,r)}(i,j) &:= D_{w_1,w_2,w_3}((s+1)i-j+r,j),\\
	\hcydje_{w_1,w_2,w_3}^{(s,r)}(i,j) &:= H_{w_1,w_2,w_3}((s+1)i-j+r,j),
\end{align*}
and in addition,
\begin{align}
	P_{w_1,w_2,w_3}^{(s,r)}(u,v) &:= \sum_{i,j\ge 0} \cydje_{w_1,w_2,w_3}^{(s,r)}(i,j) u^i v^j,\label{eq:dje-gf}\\
	\hP_{w_1,w_2,w_3}^{(s,r)}(u,v) &:= \sum_{i,j\ge 0} \hcydje_{w_1,w_2,w_3}^{(s,r)}(i,j) u^i v^j.\label{eq:hdje-gf}
\end{align}

\begin{lemma}
	For $s,r\ge 0$,
	\begin{align}\label{eq:D-sr-gf}
		&P_{w_1,w_2,w_3}^{(s,r)}(u,v)\notag\\
		&\,\,= \RP_{u} \frac{1}{s+1}\sum_{k=0}^{s} \frac{\big(e^{\frac{2\pi \ii k}{s+1}} u^{\frac{1}{s+1}}\big)^{-r}}{1-w_1\big(e^{\frac{2\pi \ii k}{s+1}} u^{\frac{1}{s+1}}\big)-w_2\big(e^{\frac{2\pi \ii k}{s+1}} u^{\frac{1}{s+1}}\big)v-w_3\big(e^{\frac{2\pi \ii k}{s+1}} u^{\frac{1}{s+1}}\big)^{2}v},
	\end{align}
	and
	\begin{align}\label{eq:H-sr-gf}
		&\hP_{w_1,w_2,w_3}^{(s,r)}(u,v)\notag\\
		&\,\,= \RP_{u} \frac{1}{s+1}\sum_{k=0}^{s} \frac{\big(e^{\frac{2\pi \ii k}{s+1}} u^{\frac{1}{s+1}}\big)^{-r} \big(w_2+w_3\big(e^{\frac{2\pi \ii k}{s+1}} u^{\frac{1}{s+1}}\big)\big)}{1-w_1\big(e^{\frac{2\pi \ii k}{s+1}} u^{\frac{1}{s+1}}\big)-w_2\big(e^{\frac{2\pi \ii k}{s+1}} u^{\frac{1}{s+1}}\big)v-w_3\big(e^{\frac{2\pi \ii k}{s+1}} u^{\frac{1}{s+1}}\big)^{2}v}.
	\end{align}
\end{lemma}

\begin{proof}
	We apply the $(s+1)$-dissection with respect to the variable $u$ to the series $\sum_{i,j\ge 0} D_{w_1,w_2,w_3}(i-j,j)u^i v^j$ and obtain that
	\begin{align*}
		&\sum_{i,j\ge 0} \cydje_{w_1,w_2,w_3}^{(s,r)}(i,j) u^i v^j\\
		&\qquad = \RP_{u} \frac{1}{s+1}\sum_{k=0}^{s}\big(e^{\frac{2\pi \ii k}{s+1}} u^{\frac{1}{s+1}}\big)^{-r}\sum_{i,j\ge 0} D_{w_1,w_2,w_3}(i-j,j) \big(e^{\frac{2\pi \ii k}{s+1}} u^{\frac{1}{s+1}}\big)^i v^j\\
		&\qquad= \RP_{u} \frac{1}{s+1}\sum_{k=0}^{s}\big(e^{\frac{2\pi \ii k}{s+1}} u^{\frac{1}{s+1}}\big)^{-r} P_{w_1,w_2,w_3}\Big(\big(e^{\frac{2\pi \ii k}{s+1}} u^{\frac{1}{s+1}}\big),\big(e^{\frac{2\pi \ii k}{s+1}} u^{\frac{1}{s+1}}\big)v\Big),
	\end{align*}
	which is \eqref{eq:D-sr-gf}. For \eqref{eq:H-sr-gf}, we use a similar argument.
\end{proof}

To determine the relations between the Delannoy and H-Delannoy determinants and our desired KKS determinants, the takeaway is a trick that has been widely utilized in the literature. In what follows, we provide an explicit form of a specialization of \cite[p.~24, Lemmas~4.1 and 4.2]{DF2021} for practical purposes.

\begin{lemma}\label{le:right-action}
	Let $F(u,v)\in \mathbb{C}[[u,v]]$. Suppose that $\alpha(v)$ and $\beta(v)$ are both in $\mathbb{C}[[v]]$ with constant term equal to $1$. Then
	\begin{align*}
		\underset{{0\le i,j\le n-1}}{\det} \big([u^i v^j] F(u,v)\big) = \underset{{0\le i,j\le n-1}}{\det} \big([u^i v^j] F^{\dagger}(u,v)\big),
	\end{align*}
	where
	\begin{align*}
		F^{\dagger}(u,v) = \alpha(v)\cdot  F\big(u,v\beta(v)\big).
	\end{align*}
\end{lemma}

\begin{proof}
	It is clear that the coefficient matrix of $\alpha(v)/(1-uv\beta(v))$ is an upper triangular matrix with diagonal entries equal to $1$. To obtain the claimed relation, we notice that by multiplying this upper triangular matrix with the coefficient matrix of $F(u,v)$ from the right, we exactly arrive at the coefficient matrix of $F^{\dagger}(u,v)$.
\end{proof}

We have a parallel result that will not be used in this work but is of independent interest.

\begin{lemma}\label{le:left-action}
	Let $F(u,v)\in \mathbb{C}[[u,v]]$. Suppose that $\alpha(u)$ and $\beta(u)$ are both in $\mathbb{C}[[u]]$ with constant term equal to $1$. Then
	\begin{align*}
		\underset{{0\le i,j\le n-1}}{\det} \big([u^i v^j] F(u,v)\big) = \underset{{0\le i,j\le n-1}}{\det} \big([u^i v^j] F^{\ddagger}(u,v)\big),
	\end{align*}
	where
	\begin{align*}
		F^{\ddagger}(u,v) = \alpha(u)\cdot  F\big(u\beta(u),v\big).
	\end{align*}
\end{lemma}

\begin{proof}
	This time we notice that the coefficient matrix of $\alpha(u)/(1-uv\beta(u))$ is lower triangular with diagonal entries equal to $1$. Now multiplying this matrix with the coefficient matrix of $F(u,v)$ from the left leads us to the desired claim.
\end{proof}

Recall that to define the KKS determinants, we have written
\begin{align*}
	\cyb_{a,b,c,d}^{(m,l)}(i,j) = l^{j+b}\binom{mi+j+c}{mi+a} + \binom{mi-j+d}{mi+a}.
\end{align*}
Let
\begin{align*}
	B_{a,b,c,d}^{(m,l)}(u,v) := \sum_{i,j\ge 0} \cyb_{a,b,c,d}^{(m,l)}(i,j) u^i v^j.
\end{align*}

\begin{lemma}\label{le:B-gf}
	For $m\ge 1$ and $a\ge 0$,
	\begin{align}\label{eq:B-gf}
		B_{a,b,c,d}^{(m,l)}(u,v) &= \RP_{u}\frac{l^b}{m} \sum_{k=0}^{m-1} \frac{\big(e^{\frac{2\pi \ii k}{m}} u^{\frac{1}{m}}\big)^{-a}}{\big(1-e^{\frac{2\pi \ii k}{m}} u^{\frac{1}{m}}\big)^{c-a} \big(1-lv-e^{\frac{2\pi \ii k}{m}} u^{\frac{1}{m}}\big)}\notag\\
		&\quad + \RP_{u} \frac{1}{m} \sum_{k=0}^{m-1} \frac{\big(e^{\frac{2\pi \ii k}{m}} u^{\frac{1}{m}}\big)^{-a}}{\big(1-e^{\frac{2\pi \ii k}{m}} u^{\frac{1}{m}}\big)^{1+d-a} \big(1-v+ve^{\frac{2\pi \ii k}{m}} u^{\frac{1}{m}}\big)}.
	\end{align}
\end{lemma}

\begin{proof}
	Note that
	\begin{align*}
		\sum_{i,j\ge 0} l^{j+b}\binom{mi+j+c}{mi+a} u^i v^j = l^b \sum_{j\ge 0} (lv)^j \sum_{i\ge 0} \binom{mi+a+j+c-a}{mi+a} u^i.
	\end{align*}
	For the inner sum, we apply the $m$-dissection to the following variant of the \emph{binomial theorem}:
	\begin{align}\label{eq:bin-gf}
		\sum_{i\ge 0} \binom{i+X}{i} u^i = \frac{1}{(1-u)^{1+X}}.
	\end{align}
	Thus,
	\begin{align*}
		\sum_{i\ge 0} \binom{mi+a+j+c-a}{mi+a} u^i = \RP_{u} \frac{1}{m} \sum_{k=0}^{m-1} \frac{\big(e^{\frac{2\pi \ii k}{m}} u^{\frac{1}{m}}\big)^{-a}}{\big(1-e^{\frac{2\pi \ii k}{m}} u^{\frac{1}{m}}\big)^{j+1+c-a}}.
	\end{align*}
	It follows that
	\begin{align*}
		&\sum_{i,j\ge 0} l^{j+b}\binom{mi+j+c}{mi+a} u^i v^j\\
		&\qquad= \RP_{u} \frac{l^b}{m} \sum_{k=0}^{m-1} \frac{\big(e^{\frac{2\pi \ii k}{m}} u^{\frac{1}{m}}\big)^{-a}}{\big(1-e^{\frac{2\pi \ii k}{m}} u^{\frac{1}{m}}\big)^{1+c-a}} \sum_{j\ge 0} \left(\frac{lv}{1-e^{\frac{2\pi \ii k}{m}} u^{\frac{1}{m}}}\right)^j\\
		&\qquad = \RP_{u} \frac{l^b}{m} \sum_{k=0}^{m-1} \frac{\big(e^{\frac{2\pi \ii k}{m}} u^{\frac{1}{m}}\big)^{-a}}{\big(1-e^{\frac{2\pi \ii k}{m}} u^{\frac{1}{m}}\big)^{c-a} \big(1-lv-e^{\frac{2\pi \ii k}{m}} u^{\frac{1}{m}}\big)}.
	\end{align*}
	In a similar way, we have
	\begin{align*}
		&\sum_{i,j\ge 0} \binom{mi-j+d}{mi+a} u^i v^j\\
		&\qquad = \RP_{u} \frac{1}{m} \sum_{k=0}^{m-1} \frac{\big(e^{\frac{2\pi \ii k}{m}} u^{\frac{1}{m}}\big)^{-a}}{\big(1-e^{\frac{2\pi \ii k}{m}} u^{\frac{1}{m}}\big)^{1+d-a} \big(1-v+ve^{\frac{2\pi \ii k}{m}} u^{\frac{1}{m}}\big)}.
	\end{align*}
	The claimed relation then follows.
\end{proof}

We close this section with the following relations, from which the second part in Theorems~\ref{th:main-d} and \ref{th:main-h} can be confirmed.

\begin{theorem}[$=$ Second Part in Theorems~\ref{th:main-d} and \ref{th:main-h}]\label{th:L-B}
	For $m\ge 1$ and $a\ge 0$,
	\begin{align}
		\cyL_{1,l-1,1}^{(m-1,a)}(n) &= \tfrac{1}{2} \cyB_{a,0,a,a}^{(m,l)}(n),\label{eq:L-B-D}\\
		\hcyL_{1,l-1,1}^{(m-1,a)}(n) &= \cyB_{a+1,1,a+1,a-1}^{(m,l)}(n).\label{eq:L-B-H}
	\end{align}
\end{theorem}

\begin{proof}
	We first show \eqref{eq:L-B-D}. Let
	\begin{align*}
		G(u,v):=\frac{u^{-a}}{1-u-(l-1)uv-u^2v}.
	\end{align*}
	A direct computation reveals that
	\begin{align*}
		&\frac{1-lv^2}{(1-v)(1-lv)}G\left(u,\frac{v}{(1-v)(1-lv)}\right)\\
		&\qquad = \frac{u^{-a}}{1-lv-u} + \frac{u^{-a}}{(1-u)(1-v+uv)} - \frac{u^{-a}}{1-u}.
	\end{align*}
	Now on both sides, we replace $u$ with $e^{\frac{2\pi \ii k}{m}} u^{\frac{1}{m}}$ for each $k$ with $0\le k\le m-1$ and sum over these $k$. Dividing the sum by $m$ and applying $\RP_{u}$, then in light of \eqref{eq:D-sr-gf}, the left-hand side becomes
	\begin{align*}
		\frac{1-lv^2}{(1-v)(1-lv)} P_{1,l-1,1}^{(m-1,a)}\left(u,\frac{v}{(1-v)(1-lv)}\right),
	\end{align*}
	while it follows from \eqref{eq:B-gf} that the right-hand side becomes
	\begin{align*}
		B_{a,0,a,a}^{(m,l)}(u,v) - \RP_{u} \frac{1}{m}\sum_{k=0}^{m-1} \frac{\big(e^{\frac{2\pi \ii k}{m}} u^{\frac{1}{m}}\big)^{-a}}{1-e^{\frac{2\pi \ii k}{m}} u^{\frac{1}{m}}}.
	\end{align*}
	Note that
	\begin{align*}
		\RP_{u} \frac{1}{m}\sum_{k=0}^{m-1} \frac{\big(e^{\frac{2\pi \ii k}{m}} u^{\frac{1}{m}}\big)^{-a}}{1-e^{\frac{2\pi \ii k}{m}} u^{\frac{1}{m}}} = \RP_{u} \sum_{j\ge -a} \frac{1}{m} \sum_{k=0}^{m-1} \big(e^{\frac{2\pi \ii k}{m}} u^{\frac{1}{m}}\big)^{j} = \sum_{\substack{j\ge 0\\m\mid j}} u^{\frac{j}{m}} = \frac{1}{1-u},
	\end{align*}
	where we have invoked the assumption that $a\ge 0$ for the second equality. Thus,
	\begin{align}\label{eq:P-B-relation}
		\frac{1-lv^2}{(1-v)(1-lv)} P_{1,l-1,1}^{(m-1,a)}\left(u,\frac{v}{(1-v)(1-lv)}\right) = B_{a,0,a,a}^{(m,l)}(u,v) - \frac{1}{1-u}.
	\end{align}
	Finally, we recall from \eqref{eq:det-D} and \eqref{eq:dje-gf},
	\begin{align*}
		\cyL_{1,l-1,1}^{(m-1,a)}(n) = \underset{{0\le i,j\le n-1}}{\det} \big(\cydje_{1,l-1,1}^{(m-1,a)}(i,j)\big) = \underset{{0\le i,j\le n-1}}{\det} \big([u^i v^j] P_{1,l-1,1}^{(m-1,a)}(u,v)\big).
	\end{align*}
	Therefore, an application of Lemma~\ref{le:right-action} gives
	\begin{align*}
		\cyL_{1,l-1,1}^{(m-1,a)}(n) = \underset{{0\le i,j\le n-1}}{\det} \left([u^i v^j] \frac{1-lv^2}{(1-v)(1-lv)} P_{1,l-1,1}^{(m-1,a)}\left(u,\frac{v}{(1-v)(1-lv)}\right)\right),
	\end{align*}
	which, with recourse to \eqref{eq:P-B-relation}, further yields that
	\begin{align*}
		\cyL_{1,l-1,1}^{(m-1,a)}(n) = \underset{{0\le i,j\le n-1}}{\det} \big(\cyb_{a,0,a,a}^{(m,l)}(i,j) - \delta_{0,j}\big),
	\end{align*}
	where $\delta_{0,j}$ is the \emph{Kronecker delta}. Note that for every $i\ge 0$,
	\begin{align*}
		\cyb_{a,0,a,a}^{(m,l)}(i,0) = l^{0+0}\binom{mi+0+a}{mi+a} + \binom{mi-0+a}{mi+a} = 2,
	\end{align*}
	so that
	\begin{align*}
		\cyb_{a,0,a,a}^{(m,l)}(i,0) - \delta_{0,0} = 2 - 1 = 1 = \tfrac{1}{2} \cyb_{a,0,a,a}^{(m,l)}(i,0).
	\end{align*}
	Hence, 
	\begin{align*}
		\underset{{0\le i,j\le n-1}}{\det} \big(\cyb_{a,0,a,a}^{(m,l)}(i,j) - \delta_{0,j}\big) = \tfrac{1}{2}\cdot \underset{{0\le i,j\le n-1}}{\det} \big(\cyb_{a,0,a,a}^{(m,l)}(i,j)\big),
	\end{align*}
	thereby giving the claimed result \eqref{eq:L-B-D}.
	
	Next, we prove \eqref{eq:L-B-H}. Let
	\begin{align*}
		\widehat{G}(u,v):=\frac{u^{-a}(l-1+u)}{1-u-(l-1)uv-u^2v}.
	\end{align*}
	This time we utilize the identity
	\begin{align*}
		&\frac{1-lv^2}{(1-v)^2(1-lv)^2}\widehat{G}\left(u,\frac{v}{(1-v)(1-lv)}\right)\\
		&\qquad = \frac{lu^{-a-1}}{1-lv-u} + \frac{u^{-a-1}}{(1-u)^{-1}(1-v+uv)} - \frac{u^{-a-1}(l+1-2lv)}{(1-v)(1-lv)},
	\end{align*}
	and invoke \eqref{eq:H-sr-gf} and \eqref{eq:B-gf} to derive that
	\begin{align*}
		&\frac{1-lv^2}{(1-v)^2(1-lv)^2} \hP_{1,l-1,1}^{(m-1,a)}\left(u,\frac{v}{(1-v)(1-lv)}\right)\\
		&\qquad = B_{a+1,1,a+1,a-1}^{(m,l)}(u,v) - \RP_{u} \frac{1}{m}\sum_{k=0}^{m-1} \frac{\big(e^{\frac{2\pi \ii k}{m}} u^{\frac{1}{m}}\big)^{-a-1}(l+1-2lv)}{(1-v)(1-lv)}.
	\end{align*}
	Since $a\ge 0$, it is clear that
	\begin{align*}
		\RP_{u} \frac{1}{m}\sum_{k=0}^{m-1} \frac{\big(e^{\frac{2\pi \ii k}{m}} u^{\frac{1}{m}}\big)^{-a-1}(l+1-2lv)}{(1-v)(1-lv)} = 0.
	\end{align*}
	Therefore,
	\begin{align*}
		\frac{1-lv^2}{(1-v)^2(1-lv)^2} \hP_{1,l-1,1}^{(m-1,a)}\left(u,\frac{v}{(1-v)(1-lv)}\right) = B_{a+1,1,a+1,a-1}^{(m,l)}(u,v).
	\end{align*}
	The claimed relation \eqref{eq:L-B-H} then follows by applying Lemma~\ref{le:right-action} and recalling \eqref{eq:det-H} and \eqref{eq:hdje-gf}.
\end{proof}

\section{Proof of two conjectural product formulas for KKS determinants}\label{sec:KKS-conj}

The \emph{apotheosis} of the present work showcases the proof of two conjectural determinants of Koutschan, Krattenthaler and Schlosser~\cite[p.~30, Conjecture~23, eqs.~(10.9) and (10.8)]{KKS2025}.

\begin{theorem}[$=$ Theorem \ref{th:det-conj}]
	For all $n\ge 1$,
	\begin{align}\label{eq:KKS-WH31}
		\underset{{0\le i,j\le n-1}}{\det} \left(2^{i+1}\binom{4j+i+2}{4j+2} + \binom{4j-i}{4j+2}\right) = \prod_{i=1}^{n} \frac{\Gamma(6i-1) \Gamma(\frac{i+2}{4})}{2(2i-1)\Gamma(5i-1) \Gamma(\frac{5i-2}{4})},
	\end{align}
	and
	\begin{align}\label{eq:KKS-WD33}
		\underset{{0\le i,j\le n-1}}{\det} \left(2^{i}\binom{4j+i+3}{4j+3} + \binom{4j-i+3}{4j+3}\right) = 2\prod_{i=1}^{n} \frac{\Gamma(6i-1) \Gamma(\frac{i+3}{4})}{\Gamma(5i) \Gamma(\frac{5i-1}{4})}.
	\end{align}
\end{theorem}

\subsection{Holonomic Ansatz}

As explained and utilized by Koutschan, Krattenthaler and Schlosser in \cite{KKS2025}, one of the most powerful techniques to evaluate determinants with holonomic entries is the principle of \emph{holonomic Ansatz} \cite{Zei2007}. (We refer the reader to \cite[pp.~4--5, Sect.~2]{KKS2025} for terminology and mathematical details.)

Generically, we are to evaluate the determinants $\det (A_{n})$ where the square matrices $A_n := (a_{i,j})_{0\le i,j< n}$ come from the leading principal submatrices of a fixed \emph{infinite} matrix $A:=(a_{i,j})_{i,j\ge 0}$ whose entries $a_{i,j}$ form holonomic sequences in the indices $i$ and $j$.

Now the takeaway of the holonomic Ansatz principle is that by defining the \emph{$(n-1,j)$-th normalized cofactors} of $A_n$:
\begin{align}\label{eq:c-n,j-def}
	c_{n,j} := (-1)^{n-1+j} \frac{M_{n-1,j}}{M_{n-1,n-1}} \qquad (0\le j\le n-1),
\end{align}
where $M_{i,j}$ is the \emph{$(i,j)$-th minor} of $A_n$, we \emph{often} have, although \emph{not universally} true, that the bivariate sequence $c_{n,j}$ is \emph{holonomic}. In practice, we can compute a dataset of $c_{n,j}$ from the infinite matrix $A$, and take advantage of Kauers' \emph{Mathematica} package \texttt{Guess} \cite{Kau2009} to predict a set of recurrences satisfied by $c_{n,j}$, which in turn, produces a \emph{(left) Gr\"obner basis} $\mathsf{annc}$ of the \emph{annihilators} for $c_{n,j}$.

It remains to show $c_{n,j}$ is indeed annihilated by $\mathsf{annc}$. We start by noting from \eqref{eq:c-n,j-def} that
\begin{align}\label{eq:H1}
	c_{n,n-1} = 1\qquad (n\ge 1),\tag{H1}
\end{align}
and from the Laplace expansion with respect to the $i$-th row of $A_n$ that
\begin{align}\label{eq:H2}
	\Sigma_{i,n}:=\sum_{j=0}^{n-1} a_{i,j} c_{n,j} = 0\qquad (0\le i < n-1).\tag{H2}
\end{align}
In the meantime, if a bivariate sequence $c'_{n,j}$ satisfies the relations \eqref{eq:H1} and \eqref{eq:H2}, then it is uniquely determined under the \emph{a priori} assumption that each $A_n$ has full rank, or equivalently, $\det (A_n) \ne 0$. (In the cases of \eqref{eq:KKS-WH31} and \eqref{eq:KKS-WD33}, since the two determinants count certain positively weighted nonintersecting lattice paths that always exist, their values are nonvanishing.) Hence, if such $c'_{n,j}$ is produced under the annihilation of the guessed $\mathsf{annc}$ together with the initial values from $c_{n,j}$, we may argue that $c_{n,j} =c'_{n,j}$ by the uniqueness so that $c_{n,j}$ is also annihilated by $\mathsf{annc}$.

To show $c'_{n,j}$ satisfies \eqref{eq:H1}, we only need to produce an annihilator for $c'_{n,n-1}$ from $\mathsf{annc}$ in light of the properties of \emph{holonomic closure} \cite[p.~130, Theorem 4]{Kau2013}. Then we examine that this annihilator admits a constant sequence as its solution so that checking a suitable number of initial values concludes the claim. This part can be quickly done.

To show $c'_{n,j}$ satisfies \eqref{eq:H2}, we apply the method of \emph{creative telescoping} \cite{Zei1991} to the summand $a_{i,j}c'_{n,j}$ to get an annihilating basis for the summation
\begin{align*}
	\Sigma'_{i,n}:= \sum_{j=0}^{n-1} a_{i,j} c'_{n,j}.
\end{align*}
Then we compute the singularities of this basis and determine for which \emph{finite} set of initial values we have to check separately.

Once we have proven the annihilators $\mathsf{annc}$ for $c_{n,j}$, we apply creative telescoping to get a recurrence for the univariate sequence
\begin{align*}
	\Sigma_{n-1,n} := \sum_{j=0}^{n-1} a_{n-1,j} c_{n,j}\qquad (n\ge 1).
\end{align*}
On the other hand, by the Laplace expansion with respect to the last row of $A_n$, it is true that
\begin{align}\label{eq:H3}
	\sum_{j=0}^{n-1} a_{n-1,j} c_{n,j} = \frac{\det (A_n)}{\det (A_{n-1})}\qquad (n\ge 1).\tag{H3}
\end{align}
Hence, as long as we can formulate a potential expression for each $\det (A_n)$, it suffices to verify that the right-hand side of \eqref{eq:H3} satisfies the same recurrence for $\Sigma_{n-1,n}$ while a suitable number of initial values also match. This potential expression therefore becomes valid.

The principle of holonomic Ansatz works perfectly for most determinants evaluated by Koutschan, Krattenthaler and Schlosser; see \cite[Theorems~3, 4, 10, 12, 14 and 15]{KKS2025}. Taking one of the most complicated cases among those, namely, \cite[p.~20, eqs.~(7.11)+(7.8)]{KKS2025}:
\begin{align}\label{eq:KKS-ex}
	\underset{{0\le i,j\le n-1}}{\det} \left(3^{i+4}\binom{3j+i+8}{3j+8} + \binom{3j-i}{3j+8}\right) = \prod_{i=1}^{n} \frac{2^{i-1}\Gamma(4i+7) \Gamma(\frac{i+1}{3})}{\Gamma(3i+6) \Gamma(\frac{4i+4}{4})},
\end{align}
as an example, the entries in the square matrices are first split as
\begin{align*}
	a_{i,j} := a_{i,j}^{(1)} + a_{i,j}^{(2)},
\end{align*}
where
\begin{align*}
	a_{i,j}^{(1)} := 3^{i+4}\binom{3j+i+8}{3j+8},\qquad a_{i,j}^{(2)} := \binom{3j-i}{3j+8}.
\end{align*}
We still use $c_{n,j}$ to denote the normalized cofactors of $A_n$, which are annihilated by a certain guessed Gr\"obner basis $\mathsf{annc}$. (Below we shall not distinguish $c_{n,j}$ and $c'_{n,j}$ by abuse of notation.) As we have discussed, the relation \eqref{eq:H1} is always easy to establish. Hence, we focus on \eqref{eq:H2} and \eqref{eq:H3}. Let
\begin{align*}
	\Sigma_{i,n}^{(1)} := \sum_{j=0}^{n-1} a_{i,j}^{(1)} c_{n,j},\qquad \Sigma_{i,n}^{(2)} := \sum_{j=0}^{n-1} a_{i,j}^{(2)} c_{n,j}.
\end{align*}

For \eqref{eq:H2}, we apply creative telescoping to deduce annihilating bases for $\Sigma_{i,n}^{(1)}$ and $\Sigma_{i,n}^{(2)}$, respectively. Then the holonomic closure properties allow us to get an annihilating basis for the desired $\Sigma_{i,n}$. In particular, such a procedure can be executed with recourse to Koutschan's \textit{Mathematica} package \texttt{HolonomicFunctions} \cite{Kou2010b} due to the algorithm in \cite{Kou2010a}. For example, the five annihilators in the Gr\"obner basis for $\Sigma_{i,n}$ are supported on
\begin{align*}
	& \{S_i^4, S_i^3, S_i^2S_n, S_iS_n^2, S_n^3, S_i^2, S_iS_n, S_n^2, S_i,S_n, 1\},\\
	& \{S_i^3S_n, S_i^3, S_i^2S_n, S_iS_n^2, S_n^3, S_i^2, S_iS_n, S_n^2, S_i,S_n, 1\},\\
	& \{S_i^2S_n^2, S_i^3, S_i^2S_n, S_iS_n^2, S_n^3, S_i^2, S_iS_n, S_n^2, S_i,S_n, 1\},\\
	& \{S_iS_n^3, S_i^3, S_i^2S_n, S_iS_n^2, S_n^3, S_i^2, S_iS_n, S_n^2, S_i,S_n, 1\},\\
	& \{S_n^4, S_i^3, S_i^2S_n, S_iS_n^2, S_n^3, S_i^2, S_iS_n, S_n^2, S_i,S_n, 1\},
\end{align*}
respectively, where $S$ represents the \emph{forward shift operator} defined by $S_l \circ f(l) := f(l+1)$ for functions $f$ in the variable $l$. These annihilators occupy $819\,944$ \textit{Mathematica} bytes in total.

For \eqref{eq:H3}, we also use the \texttt{HolonomicFunctions} package to derive an annihilator for $\Sigma_{n-1,n}$ from the annihilators for $\Sigma_{n-1,n}^{(1)}$ and $\Sigma_{n-1,n}^{(2)}$. This annihilator is supported on
\begin{align*}
	\{S_n^{10}, S_n^{9}, S_n^{8}, S_n^{7}, S_n^{6}, S_n^{5}, S_n^{4}, S_n^{3}, S_n^{2}, S_n^{1}, 1\},
\end{align*}
and occupies $200\,760$ \textit{Mathematica} bytes.

\subsection{Modular reduction}

In \emph{theory}, the same method should also apply to the determinants \eqref{eq:KKS-WH31} and \eqref{eq:KKS-WD33}. As before, the relatively easy relation \eqref{eq:H1} will not cause a problem. However, the scale of the annihilators for \eqref{eq:H2} and \eqref{eq:H3} in the two cases, as will be seen in Table~\ref{tab:ct}, is much larger than that for the previous example \eqref{eq:KKS-ex}. Therefore, in \emph{practice}, we are not able to compute these annihilators \emph{directly} by the \textsf{FindCreativeTelescoping} command offered in the \texttt{HolonomicFunctions} package within a reasonable time frame.

Fortunately, such an issue on \eqref{eq:H2} and \eqref{eq:H3} can be resolved by the trick of \emph{modular reduction} \cite[Sect.~3.4]{Kou2009}. In what follows, we present the proof of \eqref{eq:KKS-WH31} as an illustration, while the determinant evaluation \eqref{eq:KKS-WD33} can be shown analogously. All necessary computational details, including those for the relation \eqref{eq:H1}, can be found in the supplementary material \cite{CCY2025}.

While keeping other notation as before, from now on we put
\begin{align*}
	a_{i,j}^{(1)} := 2^{i+1}\binom{4j+i+2}{4j+2},\qquad a_{i,j}^{(2)} := \binom{4j-i}{4j+2}.
\end{align*}

\subsubsection{Relation \eqref{eq:H3}}

The most crucial task here is to derive annihilators for
\begin{align*}
	\Sigma_{n-1,n}^{(1)} = \sum_{j=0}^{n-1} a_{n-1,j}^{(1)} c_{n,j},\qquad \Sigma_{n-1,n}^{(2)} = \sum_{j=0}^{n-1} a_{n-1,j}^{(2)} c_{n,j},
\end{align*}
respectively, while we have already guessed the annihilators for $c_{n,j}$. Therefore, we are about to apply creative telescoping to the summands $a_{n-1,j}^{(1)} c_{n,j}$ and $a_{n-1,j}^{(2)} c_{n,j}$. In what follows, we focus on the former, and the latter can be worked out in the same way.

To apply the method of creative telescoping for $\Sigma_{n-1,n}^{(1)}$, a classical way due to Chyzak \cite{Chy2000} is to first construct an annihilating operator for the summand $a_{n-1,j}^{(1)} c_{n,j}$, which takes the form \cite[p.~260, eq.~(2.1)]{Kou2010a}:
\begin{align}\label{eq:creative-telescoping-short}
	\boldsymbol{P} + (S_j-1)\boldsymbol{Q} := \sum P_{\alpha_n}(n)S_n^{\alpha_n} + (S_j-1) \sum Q_{\beta_n,\beta_j}(n;j) S_{n}^{\beta_n}S_{j}^{\beta_j},
\end{align}
where all $P$ are rational functions in $\mathbb{Q}(n)$ and all $Q$ are rational functions in $\mathbb{Q}(n,j)$. The first part $\boldsymbol{P}$, namely, the summation involving the $P$-terms, is called the \emph{telescoper part}, and the remaining part $(S_j-1)\boldsymbol{Q}$ is called the \emph{delta part} of this annihilating operator. In \cite{Kou2010a}, Koutschan offered a faster approach by further making an Ansatz on the delta part and specifying \eqref{eq:creative-telescoping-short} as \cite[p.~261, eq.~(2.3)]{Kou2010a}:
\begin{align}\label{eq:creative-telescoping}
	\sum P_{\alpha_n}(n)S_n^{\alpha_n} + (S_j-1) \sum \frac{\tilde{Q}_{\beta_n,\mu,\nu}(n)\,j^\mu}{D_{\beta_n,\nu}(n;j)} S_{n}^{\beta_n}S_{j}^{\nu}.
\end{align}
Now all $\tilde{Q}$ are rational functions in $\mathbb{Q}(n)$ and all $D$ are polynomials in $\mathbb{Q}(n)[j]$. In particular, these denominator polynomials $D$ are highly predictable. The \textsf{FindCreativeTelescoping} command in the \texttt{HolonomicFunctions} package \cite{Kou2010b} provides an \emph{automatic} way to compute this annihilating operator \eqref{eq:creative-telescoping}. Finally, with \eqref{eq:creative-telescoping-short} or \eqref{eq:creative-telescoping} constructed, we have \cite[p.~14]{Kou2010b}:
\begin{align}\label{eq:inhomogeneous}
	\boldsymbol{P}\circ \sum_{j=0}^{n-1} a_{n-1,j}^{(1)} c_{n,j} = \left[\boldsymbol{Q}\circ a_{n-1,j}^{(1)} c_{n,j}\right]_{j=0} - \left[\boldsymbol{Q}\circ a_{n-1,j}^{(1)} c_{n,j}\right]_{j=n}.
\end{align}
The right-hand side of the above is called the \emph{inhomogeneous part}, and it evaluates to zero here (and at many other places especially for summations with \emph{natural} bounds). So we may conclude that the telescoper part annihilates $\Sigma_{n-1,n}^{(1)}$.

Since an automatic computation of \eqref{eq:creative-telescoping} becomes unrealistic in our case, we make use of the method of \emph{modular reduction}, which requires a considerable amount of \emph{manual} input. The basic idea is that we first make an \emph{Ansatz} on what form the annihilating operator \eqref{eq:creative-telescoping}, especially its telescoper part, could be like. By certain numerical experiments, we assume \emph{a priori} that the support of the telescoper is
\begin{align}\label{eq:Supp-H3}
	\Supp_{\rH_3} := \{S_n^{20}, S_n^{19}, \ldots, S_n^2, S_n, 1\}.
\end{align}
The size of this support is much bigger in comparison with that for \eqref{eq:KKS-ex}. This is the first reason why \eqref{eq:creative-telescoping} cannot be calculated automatically, especially by noting that such an automatic computation for \eqref{eq:KKS-ex} already takes up to \emph{half a day} according to \cite[p.~18, Table~1]{KKS2025}.

Once the support of the telescoper has been assumed, we use the ``\textit{FindSupport}'' mode of Koutschan's \textsf{FindCreativeTelescoping} command to construct the delta part. It is notable that the polynomials $D$ will be predicted automatically, and hence the unknown \emph{parameters} to be determined are the rational functions $P$ and $\tilde{Q}$. In particular, there are $21$ (which is exactly the size of the support $\Supp_{\rH_3}$) parameters in the telescoper part, and $691$ parameters in the delta part (after a full expansion).

Our next task is to formulate these unknown parameters.

Taking an arbitrary parameter $P(n)$ in the telescoper part as an example, we first keep in mind that it is a rational function in $\mathbb{Q}(n)$. Now we choose a set of $n$-values, say $\{n_1,\ldots,n_L\}$, and a set of \emph{large} primes, say $\{p_1,\ldots,p_M\}$. The ``\textit{Modular}'' mode of Koutschan's \textsf{FindCreativeTelescoping} command produces $L\times M$ interpolated values
\begin{align*}
	P(n_l) \bmod p_m.
\end{align*}
By the classical approach of \emph{rational reconstruction} \cite[Sect.~5.7]{vzGG2013}, we may recover this rational function $P(n)$ from the interpolations. In practice, we choose $L=1296$ different $n$-values and $M=300$ different large primes of size around $2^{32}$. The latter is especially necessary for rational reconstruction because the largest coefficient among the whole telescoper part has $977$ digits! For the size of required interpolations for other quantities in our proof of \eqref{eq:KKS-WH31} and \eqref{eq:KKS-WD33}, we refer the reader to Table~\ref{tab:interpolation}.

\begin{table}[ht!]
	\def\arraystretch{1.5}
	\centering
	\caption{Size of the required interpolations}\label{tab:interpolation}
	\begin{tabular}{p{0.15\textwidth}|>{\centering\arraybackslash}p{0.11\textwidth}|>{\centering\arraybackslash}p{0.11\textwidth}|>{\centering\arraybackslash}p{0.11\textwidth}|>{\centering\arraybackslash}p{0.11\textwidth}|>{\centering\arraybackslash}p{0.11\textwidth}|>{\centering\arraybackslash}p{0.11\textwidth}}
		\hline
		\cellcolor{gray!20}Det.~\eqref{eq:KKS-WH31} & \multicolumn{2}{c|}{\eqref{eq:H2}, \#1} & \multicolumn{2}{c|}{\eqref{eq:H2}, \#2} & \multicolumn{2}{c}{\eqref{eq:H3}}\\
		\hline
		& Smnd 1 & Smnd 2 & Smnd 1 & Smnd 2 & Smnd 1 & Smnd 2\\
		\hline
		\# of $i$ & $160$ & $160$ & $160$ & $160$ & --- & ---\\
		\# of $n$ & $528$ & $528$ & $528$ & $528$ & $1296$ & $1296$\\
		\# of primes & $40$ & $40$ & $40$ & $40$ & $300$ & $300$\\
		\hline
		\multicolumn{7}{l}{}\\
		\hline
		\cellcolor{gray!20}Det.~\eqref{eq:KKS-WD33} & \multicolumn{2}{c|}{\eqref{eq:H2}, \#1} & \multicolumn{2}{c|}{\eqref{eq:H2}, \#2} & \multicolumn{2}{c}{\eqref{eq:H3}}\\
		\hline
		& Smnd 1 & Smnd 2 & Smnd 1 & Smnd 2 & Smnd 1 & Smnd 2\\
		\hline
		\# of $i$ & $200$ & $200$ & $200$ & $200$ & --- & ---\\
		\# of $n$ & $528$ & $528$ & $528$ & $528$ & $1296$ & $1296$\\
		\# of primes & $40$ & $40$ & $40$ & $40$ & $300$ & $300$\\
		\hline
	\end{tabular}
\end{table}

After working out the rational reconstruction for all unknown parameters in the telescoper and delta parts, we arrive at a yet postulated annihilating operator \eqref{eq:creative-telescoping} for $a_{n-1,j}^{(1)} c_{n,j}$, and this can be finally certified by the \textsf{OreReduce} command of the \texttt{HolonomicFunctions} package after inputting the information of $a_{n-1,j}^{(1)} c_{n,j}$.

The above computation for modular reduction was executed in \emph{parallel} on an Amazon cluster (AWS EC2, instance type: m5.24xlarge\footnote{See \url{https://instances.vantage.sh/aws/ec2/m5.24xlarge} for details.}, 96 CPUs, 384 GB Memory\footnote{We report that a Memory overflow problem was always encountered due to the scale of our computation at our initial attempts so eventually we had to launch only 48 kernels for interpolations and 72 kernels for rational reconstruction.}). In particular, it took around \emph{1.5 Days} to generate interpolating points and another \emph{2.5 Days} for the rational reconstruction of all $21+691=712$ parameters. The outputting telescoper part occupies $6\,102\,528$ \textit{Mathematica} bytes and the delta part occupies $3\,182\,669\,576$ bytes. The scale of them, which is extraordinarily huge, gives the second (and the most critical) reason why an automatic application of Koutschan's \textsf{FindCreativeTelescoping} command is impractical. For the data of other calculations, we refer the reader to Table~\ref{tab:ct}.

\begin{table}[ht!]
	\def\arraystretch{1.5}
	\centering
	\caption{Data for the modular reduction of creative telescoping}\label{tab:ct}
	\begin{tabular}{p{0.15\textwidth}|>{\centering\arraybackslash}p{0.11\textwidth}|>{\centering\arraybackslash}p{0.11\textwidth}|>{\centering\arraybackslash}p{0.11\textwidth}|>{\centering\arraybackslash}p{0.11\textwidth}|>{\centering\arraybackslash}p{0.11\textwidth}|>{\centering\arraybackslash}p{0.11\textwidth}}
		\hline
		\cellcolor{gray!20}Det.~\eqref{eq:KKS-WH31} & \multicolumn{2}{c|}{\eqref{eq:H2}, \#1} & \multicolumn{2}{c|}{\eqref{eq:H2}, \#2} & \multicolumn{2}{c}{\eqref{eq:H3}}\\
		\hline
		& Smnd 1 & Smnd 2 & Smnd 1 & Smnd 2 & Smnd 1 & Smnd 2\\
		\hline
		\multicolumn{7}{l}{\textit{Telescoper part}}\\
		\hline
		Support & \multicolumn{2}{c|}{$\Supp_{\rH_2}^\dagger$ in \eqref{eq:Supp-H2-1}} & \multicolumn{2}{c|}{$\Supp_{\rH_2}^\ddagger$ in \eqref{eq:Supp-H2-2}} & \multicolumn{2}{c}{$\Supp_{\rH_3}$ in \eqref{eq:Supp-H3}}\\
		\hline
		\# of paras. & $21$ & $21$ & $21$ & $21$ & $21$ & $21$\\
		ByteCount & {\footnotesize $27\,393\,808$} & {\footnotesize $27\,393\,808$} & {\footnotesize $21\,764\,496$} & {\footnotesize $20\,869\,424$} & {\footnotesize $6\,102\,528$} & {\footnotesize $5\,638\,048$}\\
		\hline
		\multicolumn{7}{l}{\textit{Delta part}}\\
		\hline
		\# of paras. & $148$ & $166$ & $123$ & $137$ & $691$ & $595$\\
		ByteCount & {\tiny $1\,777\,395\,424$} & {\tiny $2\,657\,465\,696$} & {\tiny $1\,033\,241\,824$} & {\tiny $1\,558\,607\,768$} & {\tiny $3\,182\,669\,576$} & {\tiny $4\,665\,533\,376$}\\
		\hline
		\multicolumn{7}{l}{}\\
		\hline
		\cellcolor{gray!20}Det.~\eqref{eq:KKS-WD33} & \multicolumn{2}{c|}{\eqref{eq:H2}, \#1} & \multicolumn{2}{c|}{\eqref{eq:H2}, \#2} & \multicolumn{2}{c}{\eqref{eq:H3}}\\
		\hline
		& Smnd 1 & Smnd 2 & Smnd 1 & Smnd 2 & Smnd 1 & Smnd 2\\
		\hline
		\multicolumn{7}{l}{\textit{Telescoper part}}\\
		\hline
		Support & \multicolumn{2}{c|}{$\Supp_{\rH_2}^\dagger$ in \eqref{eq:Supp-H2-1}} & \multicolumn{2}{c|}{$\Supp_{\rH_2}^\ddagger$ in \eqref{eq:Supp-H2-2}} & \multicolumn{2}{c}{$\Supp_{\rH_3}$ in \eqref{eq:Supp-H3}}\\
		\hline
		\# of paras. & $21$ & $21$ & $21$ & $21$ & $21$ & $21$\\
		ByteCount & {\footnotesize $39\,571\,320$} & {\footnotesize $36\,878\,968$} & {\footnotesize $31\,668\,024$} & {\footnotesize $28\,931\,264$} & {\footnotesize $6\,530\,808$} & {\footnotesize $6\,029\,816$}\\
		\hline
		\multicolumn{7}{l}{\textit{Delta part}}\\
		\hline
		\# of paras. & $173$ & $191$ & $148$ & $162$ & $716$ & $620$\\
		ByteCount & {\tiny $3\,876\,381\,000$} & {\tiny $5\,048\,293\,960$} & {\tiny $2\,413\,087\,160$} & {\tiny $3\,177\,278\,424$} & {\tiny $3\,827\,584\,248$} & {\tiny $5\,432\,429\,512$}\\
		\hline
	\end{tabular}
\end{table}

Recall that the telescoper part in \eqref{eq:creative-telescoping} indeed gives an annihilator for $\Sigma_{n-1,n}^{(1)}$ as we can easily check that the inhomogeneous part vanishes. By the same application of modular reduction for creative telescoping, we also obtain an annihilator for $\Sigma_{n-1,n}^{(2)}$. More importantly, the two annihilators differ by a fixed scalar. Hence, either of them annihilates $\Sigma_{n-1,n}$, as desired.

Finally, we check that, after substituting the product formula in \eqref{eq:KKS-WH31}, the right-hand side of \eqref{eq:H3} is also annihilated by the annihilator for $\Sigma_{n-1,n}$. Meanwhile, this annihilator has no singularity when $n>20$. Considering the support $\Supp_{\rH_3}$, which has leading exponent $20$, it would be sufficient to verify the initial values at $1\le n\le 20$ for
\begin{align*}
	\sum_{j=0}^{n-1} a_{n-1,j} c_{n,j} = \frac{\det (A_n)}{\det (A_{n-1})} = \frac{\Gamma(6n-1) \Gamma(\frac{n+2}{4})}{2(2n-1)\Gamma(5n-1) \Gamma(\frac{5n-2}{4})}
\end{align*}
in order to prove \eqref{eq:H3}, and hence \eqref{eq:KKS-WH31}, for all $n\ge 1$.

\subsubsection{Relation \eqref{eq:H2}}

This time we are supposed to construct annihilators for the bivariate sequences
\begin{align*}
	\Sigma_{i,n}^{(1)} = \sum_{j=0}^{n-1} a_{i,j}^{(1)} c_{n,j},\qquad \Sigma_{i,n}^{(2)} = \sum_{j=0}^{n-1} a_{i,j}^{(2)} c_{n,j},
\end{align*}
where $0\le i< n-1$. As before, we apply creative telescoping under modular reduction for the two summands $a_{i,j}^{(1)} c_{n,j}$ and $a_{i,j}^{(2)} c_{n,j}$, respectively, and we will still use the former as an instance.

The to-be-determined annihilators for $a_{i,j}^{(1)} c_{n,j}$ have the form
\begin{align}\label{eq:creative-telescoping-H2}
	\sum P_{\alpha_i,\alpha_n}(i,n)S_i^{\alpha_i}S_n^{\alpha_n} + (S_j-1) \sum \frac{\tilde{Q}_{\beta_i,\beta_n,\mu,\nu}(i,n)\,j^\mu}{D_{\beta_i,\beta_n,\nu}(i,n;j)} S_{i}^{\beta_i}S_{n}^{\beta_n}S_{j}^{\nu},
\end{align}
where all $P$ and $\tilde{Q}$ are rational functions in $\mathbb{Q}(i,n)$ and all $D$ are polynomials in $\mathbb{Q}(i,n)[j]$.

We make the \emph{first Ansatz} that the telescoper part is supported on
\begin{align}\label{eq:Supp-H2-1}
	\Supp_{\rH_2}^\dagger := \{&S_i^5, S_i^4S_n, S_i^3S_n^2, S_i^2S_n^3, S_iS_n^4, S_n^5, S_i^4, S_i^3S_n, S_i^2S_n^2,\notag\\
	& S_iS_n^3, S_n^4, S_i^3, S_i^2S_n, S_iS_n^2, S_n^3, S_i^2, S_iS_n, S_n^2, S_i,S_n, 1\}.
\end{align}
Then an application of the ``\textit{FindSupport}'' mode of the \textsf{FindCreativeTelescoping} command decides the shape of the delta part and the associated denominator polynomials $D$, and we are left to determine the $21$ (which is exactly the size of the support $\Supp_{\rH_2}^\dagger$) unknown parameters in the telescoper part, and $148$ parameters in the delta part. Note that these parameters are rational functions in $i$ and $n$. Still using a certain $P(i,n)$ in the telescoper part as an example, we are about to choose a set of $i$-values, say $\{i_1,\ldots,i_K\}$, a set of $n$-values, say $\{n_1,\ldots,n_L\}$, and a set of \emph{large} primes, say $\{p_1,\ldots,p_M\}$, and then utilize the ``\textit{Modular}'' mode of Koutschan's \textsf{FindCreativeTelescoping} command to produce $K\times L\times M$ interpolated values
\begin{align*}
	P(i_k,n_l) \bmod p_m.
\end{align*}
In practice, we choose $K=160$ different $i$-values, $L=528$ different $n$-values, and $M=40$ different large primes of size around $2^{32}$. Finally, the rational function $P(i,n)$ can be reconstructed by these interpolations. As soon as we obtain all unknown parameters in the telescoper and delta parts, we arrive at a conjectural annihilating operator \eqref{eq:creative-telescoping-H2} for $a_{i,j}^{(1)} c_{n,j}$, which can be certified by the \textsf{OreReduce} command of the \texttt{HolonomicFunctions} package. The whole computation was executed \emph{parallelly} on our Amazon cluster, and it took around \emph{1.5 Days} to generate interpolating points and another \emph{1 Day} for the rational reconstruction of all $21+148=169$ parameters. The outputting telescoper part occupies $27\,393\,808$ \textit{Mathematica} bytes and the delta part occupies $1\,777\,395\,424$ bytes. We remark that by the theory of creative telescoping, the telescoper part annihilates $\Sigma_{i,n}^{(1)}$. In the same way, we may also obtain an annihilator for $\Sigma_{i,n}^{(2)}$ supported on $\Supp_{\rH_2}^\dagger$. It is easy to check that they differ by a constant scalar, and hence either of them annihilates $\Sigma_{i,n}$. We denote such an annihilator for $\Sigma_{i,n}$ by $\mathsf{annH2a}$.

We execute the same procedure by making the \emph{second Ansatz} that the telescoper part is supported on
\begin{align}\label{eq:Supp-H2-2}
	\Supp_{\rH_2}^\ddagger := \{&S_i^6, S_i^5, S_i^4S_n, S_i^3S_n^2, S_i^2S_n^3, S_iS_n^4, S_i^4, S_i^3S_n, S_i^2S_n^2,\notag\\
	& S_iS_n^3, S_n^4, S_i^3, S_i^2S_n, S_iS_n^2, S_n^3, S_i^2, S_iS_n, S_n^2, S_i,S_n, 1\}.
\end{align}
By modular reduction, we also have an annihilator $\mathsf{annH2b}$ for $\Sigma_{i,n}$ supported on $\Supp_{\rH_2}^\ddagger$.

Now we should point out that when Koutschan, Krattenthaler and Schlosser worked on \eqref{eq:KKS-ex}, a \emph{full} annihilating basis for $\Sigma_{i,n}$ was computed. The merit is that the set of singularities of these annihilators is \emph{finite}. However, doing so would be extremely time-consuming in our case.

If instead only the two annihilators $\mathsf{annH2a}$ and $\mathsf{annH2b}$ are known, we have to give up the benefit of a finite check for \eqref{eq:H2} and work on an \emph{infinite} set of singularities. To begin with, we define the region in our case by
\begin{align*}
	\RR := \big\{(i,n)\in \mathbb{N}^2: 0\le i < n-1\big\}.
\end{align*}
Taking the supports $\Supp_{\rH_2}^\dagger$ and $\Supp_{\rH_2}^\ddagger$ into consideration, we may translate the annihilators $\mathsf{annH2a}$ and $\mathsf{annH2b}$ into recurrences for $\Sigma_{i,n}$ restricted on $\RR$:
\begin{align}
	\gamma_{5,0}^\dagger(i,n)\cdot \Sigma_{i,n} &= \sum_{(a_1,b_1)\in \Ind_{\rH_2}^{\dagger}} \gamma_{a_1,b_1}^\dagger(i,n)\cdot \Sigma_{i-5+a_1,n+b_1},\\
	\gamma_{6,0}^\ddagger(i,n)\cdot \Sigma_{i,n} &= \sum_{(a_2,b_2)\in \Ind_{\rH_2}^{\ddagger}} \gamma_{a_2,b_2}^\ddagger(i,n)\cdot \Sigma_{i-6+a_2,n+b_2},
\end{align}
where $\gamma^\dagger(i,n)$ and $\gamma^\ddagger(i,n)$ are polynomials in $i$ and $n$, and the index sets are given by
\begin{align*}
	\Ind_{\rH_2}^{\dagger} &:= \big\{(a,b):S_i^aS_n^b\in \Supp_{\rH_2}^\dagger\backslash\{S_i^5\}\big\},\\
	\Ind_{\rH_2}^{\ddagger} &:= \big\{(a,b):S_i^aS_n^b\in \Supp_{\rH_2}^\ddagger\backslash\{S_i^6\}\big\}.
\end{align*}
We then verify that at least one of $\gamma_{5,0}^\dagger(i,n)$ and $\gamma_{6,0}^\ddagger(i,n)$ is nonvanishing at each point in the subregion $\big\{(i,n)\in \RR: i\ge 6\big\}$, and $\gamma_{5,0}^\dagger(i,n)$ is further nonvanishing in the subregion $\big\{(i,n)\in \RR: i= 5\big\}$. Furthermore, whenever $(i,n)\in \RR$ is such that $i\ge 5$, we always have $(i-5+a_1,n+b_1)\in \RR$ for all $(a_1,b_1)\in \Ind_{\rH_2}^{\dagger}$ while $i-5+a_1<i$; similarly, whenever $(i,n)\in \RR$ is such that $i\ge 6$, we always have $(i-6+a_2,n+b_2)\in \RR$ for all $(a_2,b_2)\in \Ind_{\rH_2}^{\ddagger}$ while $i-6+a_2<i$.

Hence, to show \eqref{eq:H2},
\begin{align*}
	\Sigma_{i,n}=\sum_{j=0}^{n-1} a_{i,j} c_{n,j} = 0\qquad (0\le i < n-1),
\end{align*}
it suffices to show $\Sigma_{i,n} = 0$ for all points in $\RR$ with $0\le i\le 4$. This verification, although still on an \emph{infinite} set of points, becomes feasible because we can work separately on $\Sigma_{0,n},\ldots,\Sigma_{4,n}$ with each treated as a univariate sequence. In theory, a direct application of creative telescoping by Koutschan's \textsf{FindCreativeTelescoping} command allows us to finalize the computation. However, we unfortunately notice that for the five cases of $i$, the inhomogeneous parts might be \emph{nonvanishing}. This situation makes our verification somehow trickier.

From now on, let us fix $i_0$ with $0\le i_0\le 4$. An application of creative telescoping reveals that the telescoper part equals $1$. Then the desired annihilator for $\Sigma_{i_0,n}^{(1)}$ is formulated as
\begin{align*}
	1+ (S_k-1)\boldsymbol{Q},
\end{align*}
where $\boldsymbol{Q}$ is in the operators $S_n$ and $S_j$. According to \eqref{eq:inhomogeneous}, we can express $\Sigma_{i_0,n}^{(1)}$ as
\begin{align*}
	\Sigma_{i_0,n}^{(1)} = \left[\boldsymbol{Q}\circ a_{i_0,j}^{(1)} c_{n,j}\right]_{j=0} - \left[\boldsymbol{Q}\circ a_{i_0,j}^{(1)} c_{n,j}\right]_{j=n}.
\end{align*}
In particular, we find that $\boldsymbol{Q}$ is supported on
\begin{align*}
	\{S_nS_j, S_j^2, S_n, S_j, 1\}.
\end{align*}
Using the facts that $c_{n,n-1}=1$ from \eqref{eq:H1} and that $c_{n,j}=0$ whenever $j\ge n$ by definition, we find that the inhomogeneous part vanishes at $j=n$. For the inhomogeneous part at $j=0$, we split it according to the above support of $\boldsymbol{Q}$ (while noting that $a_{i_0,j}^{(1)}$ is a polynomial in $j$):
\begin{align*}
	\left[\boldsymbol{Q}\circ a_{i_0,j}^{(1)} c_{n,j}\right]_{j=0} = \sum_{b=0}^2 O_b \circ c_{n,b},
\end{align*}
where each $O_b$ with $0\le b\le 2$ is a certain Ore polynomial in the operator $S_n$. Hence,
\begin{align}
	\Sigma_{i_0,n}^{(1)} = \sum_{b=0}^2 O_b \circ c_{n,b}.
\end{align}

Let $\mathsf{annc}_b$ denote the annihilating ideal for the \emph{univariate} sequence $c_{n,b}$ for each $b$ with $0\le b\le 2$. They can be computed from the annihilating ideal $\mathsf{annc}$ for the \emph{bivariate} sequence $c_{n,j}$ in light of the holonomic closure properties.

Now our goal is to show that $\Sigma_{i_0,n}^{(1)}$ is annihilated by $\mathsf{annc}_0$.

To see this, we apply $\mathsf{annc}_0$ to $\Sigma_{i_0,n}^{(1)}$ and show that the resulting sequence evaluates to zero for every $n\ge i_0+2$. Again, by the holonomic closure properties, we have
\begin{align*}
	\mathsf{annc}_0\circ \Sigma_{i_0,n}^{(1)} = \sum_{b=0}^2 \tilde{O}_b \circ c_{n,b},
\end{align*}
for certain new Ore polynomials $\tilde{O}_b$. We may then compute an annihilating ideal for each $\tilde{O}_b \circ c_{n,b}$ from $\mathsf{annc}_b$, and hence an annihilating ideal for their linear combination. In particular, this ideal annihilates $\mathsf{annc}_0\circ \Sigma_{i_0,n}^{(1)}$. Now we only need to verify $\mathsf{annc}_0\circ \Sigma_{i_0,n}^{(1)} = 0$ for a suitable number of initial values after working out the singularities. Then it is safe to conclude that $\mathsf{annc}_0\circ \Sigma_{i_0,n}^{(1)} = 0$ for all $n\ge i_0+2$.

Finally, we notice that $a_{i_0,j}^{(2)} = \binom{4j-i_0}{4j+2}$ equals $\binom{-i_0}{2}$ when $j=0$ and $0$ whenever $j\ge 1$. Hence,
\begin{align*}
	\Sigma_{i_0,n}^{(2)} = \sum_{j=0}^{n-1} a_{i_0,j}^{(2)} c_{n,j} = \binom{-i_0}{2} c_{n,0}.
\end{align*}
It follows that $\Sigma_{i_0,n}^{(2)}$ is also annihilated by $\mathsf{annc}_0$.

As a consequence, $\mathsf{annc}_0$ annihilates $\Sigma_{i_0,n}=\Sigma_{i_0,n}^{(1)}+\Sigma_{i_0,n}^{(2)}$. After working out the singularities, we check $\Sigma_{i_0,n}=0$ for a certain number of initial values to conclude the desired vanishing of $\Sigma_{i_0,n}$ for all $n\ge i_0+2$.

\section{Epilogue}

Before bringing down the \emph{curtain}, we revisit the question of Koutschan, Krattenthaler and Schlosser, and think about other connections between KKS determinants and our weighted enumerations of domino tilings, $\cyD_{w_1,w_2,w_3}^{(s,r)}(n)$, or equivalently, nonintersecting lattice paths, $\cyL_{w_1,w_2,w_3}^{(s,r)}(n)$. Based on a brief computer search for relations concerning $\cyL_{w_1,w_2,w_3}^{(s,r)}(n)$ and $\cyB_{a,b,c,d}^{(m,l)}(n)$, we find that apart from \eqref{eq:main-d}, it seems necessary to require one of $w_2$ and $w_3$ to be zero when we set $s$ and $r$ \emph{generic}. Note that the case where $w_1=0$ is not at all interesting because there always exists an east step in the Delannoy path from $(-n,n)$ to $(r-n,n)$ for $r\ge 1$ so that the weighted contribution from the associated system of nonintersecting Delannoy paths becomes zero. In view of Remark~\ref{rmk:c1c2-scalar}, for the moment we only look at $\cyL_{1,1,0}^{(s,r)}(n)$ and $\cyL_{1,0,1}^{(s,r)}(n)$.

\begin{theorem}
	For $s\ge 1$ and $r\ge 0$,
	\begin{align}\label{eq:L-110}
		\cyL_{1,1,0}^{(s,r)}(n) = \tfrac{1}{2} \cyB_{0,0,c,c}^{(1,s+2)}(n) = (s+1)^{\binom{n}{2}}.
	\end{align}
	In addition, letting $\varrho\ge 1$ be a positive integer, then
	\begin{align}\label{eq:L-101}
		\cyL_{1,0,1}^{(s,r)}(n) = \cyB_{r+2\varrho,b,r+2\varrho,r}^{(s+1,1)}(n) = \tfrac{1}{2} \cyB_{r,b,r,r}^{(s+1,1)}(n).
	\end{align}
\end{theorem}

\begin{proof}[Proof of \eqref{eq:L-110}]
	For the evaluation of $\cyB_{0,0,c,c}^{(1,s+2)}(n)$, we directly make use of a result of Koutschan, Krattenthaler and Schlosser~\cite[p.~6, Theorem~3]{KKS2025} to get
	\begin{align*}
		\cyB_{0,0,c,c}^{(1,s+2)}(n) = \underset{{0\le i,j\le n-1}}{\det} \left((s+2)^{j}\binom{i+j+c}{i} + \binom{i-j+c}{i}\right) = 2(s+1)^{\binom{n}{2}}.
	\end{align*}
	
	For the evaluation of $\cyL_{1,1,0}^{(s,r)}(n)$, we first recall from \eqref{eq:det-D} that
	\begin{align*}
		\cyL_{1,1,0}^{(s,r)}(n) = \underset{{0\le i,j\le n-1}}{\det} \big(D_{1,1,0}((s+1)i-j+r,j)\big).
	\end{align*}
	In addition,
	\begin{align*}
		D_{1,1,0}((s+1)i-j+r,j) = \binom{(s+1)i+r}{j},
	\end{align*}
	which follows from either the double series expansion of \eqref{eq:D-gf} or the fact that $D_{1,1,0}((s+1)i-j+r,j)$ indeed counts the number of paths from $(0,0)$ to $((s+1)i-j+r,j)$ by using only east and north steps. Therefore,
	\begin{align*}
		\cyL_{1,1,0}^{(s,r)}(n) = \underset{{0\le i,j\le n-1}}{\det} \left(\binom{(s+1)i+r}{j}\right).
	\end{align*}
	Finally, we utilize a formula of Gessel and Viennot~\cite[p.~308]{GV1985} to conclude that
	\begin{align*}
		\underset{{0\le i,j\le n-1}}{\det} \left(\binom{(s+1)i+r}{j}\right) = \frac{\displaystyle \prod_{0\le i<j \le n-1} (s+1)(j-i)}{\displaystyle \prod_{0\le k\le n-1} k!} = (s+1)^{\binom{n}{2}},
	\end{align*}
	as claimed.
\end{proof}

\begin{proof}[Proof of \eqref{eq:L-101}]
	Note that
	\begin{align*}
		\cyB_{a,b,a,d}^{(m,1)}(n) = \underset{{0\le i,j\le n-1}}{\det} \left(\binom{mi+j+a}{mi+a} + \binom{mi-j+d}{mi+a}\right),
	\end{align*}
	so that the value of $b$ does not affect this determinant. Therefore, the relation
	\begin{align*}
		\cyL_{1,0,1}^{(s,r)}(n) = \tfrac{1}{2} \cyB_{r,b,r,r}^{(s+1,1)}(n)
	\end{align*}
	simply becomes the $l=1$ case of \eqref{eq:main-d}.
	
	Now it remains to show that for $a\ge 0$ and $\varrho\ge 1$,
	\begin{align}\label{eq:a+2delta}
		\cyB_{a+2\varrho,b,a+2\varrho,a}^{(m,1)}(n) = \tfrac{1}{2} \cyB_{a,b,a,a}^{(m,1)}(n).
	\end{align}
	Let us introduce the following infinite matrix
	\begin{align*}
		\BB_{a,b,c,d}^{(m,l)} := \big(\cyb_{a,b,c,d}^{(m,l)}(i,j)\big)_{i,j\ge 0},
	\end{align*}
	where
	\begin{align*}
		\cyb_{a,b,c,d}^{(m,l)}(i,j) = l^{j+b}\binom{mi+j+c}{mi+a} + \binom{mi-j+d}{mi+a}.
	\end{align*}
	In what follows, we shall show the matrix identity
	\begin{align}
		\BB_{a,b,a,a}^{(m,1)} - \big(\varepsilon(i,j)\big)_{i,j\ge 0} = \BB_{a+2\varrho,b,a+2\varrho,a}^{(m,1)}\cdot \big((-1)^{j-i}\tbinom{2\varrho}{j-i}\big)_{i,j\ge 0},
	\end{align}
	where $\varepsilon(i,j) = 1$ when $j=0$, and $\varepsilon(i,j) = 0$ otherwise. This relation clearly implies \eqref{eq:a+2delta} because
	\begin{align*}
		\underset{{0\le i,j\le n-1}}{\det} \big((-1)^{j-i}\tbinom{2\varrho}{j-i}\big) = 1,
	\end{align*}
	while for $i\ge 0$,
	\begin{align*}
		\cyb_{a,b,a,a}^{(m,1)}(i,j) - \varepsilon(i,j) = \begin{cases}
			\tfrac{1}{2}\cyb_{a,b,a,a}^{(m,1)}(i,j), & \text{if $j=0$},\\[6pt]
			\cyb_{a,b,a,a}^{(m,1)}(i,j), & \text{if $j\ge 1$}.
		\end{cases}
	\end{align*}
	
	For the moment, let us write
	\begin{align*}
		\big(\tcyb_{a+2\varrho,b,a+2\varrho,a}^{(m,1)}(i,j)\big)_{i,j\ge 0} := \BB_{a+2\varrho,b,a+2\varrho,a}^{(m,1)}\cdot \big((-1)^{j-i}\tbinom{2\varrho}{j-i}\big)_{i,j\ge 0}.
	\end{align*}
	We notice that
	\begin{align*}
		\sum_{i,j\ge 0} \tcyb_{a+2\varrho,b,a+2\varrho,a}^{(m,1)}(i,j) u^i v^j = (1-v)^{2\varrho} \sum_{i,j\ge 0} \cyb_{a+2\varrho,b,a+2\varrho,a}^{(m,1)}(i,j) u^i v^j.
	\end{align*}
	Now it suffices to prove
	\begin{align}\label{eq:rho-gf}
		\sum_{i,j\ge 0} \cyb_{a,b,a,a}^{(m,1)}(i,j) u^i v^j - \frac{1}{1-u} = (1-v)^{2\varrho} \sum_{i,j\ge 0} \cyb_{a+2\varrho,b,a+2\varrho,a}^{(m,1)}(i,j) u^i v^j.
	\end{align}
	To do so, we establish two separate identities \eqref{eq:rho-gf-1} and \eqref{eq:rho-gf-2}, which imply \eqref{eq:rho-gf} jointly.
	
	First, we show that
	\begin{align}\label{eq:rho-gf-1}
		\sum_{i,j\ge 0} \binom{mi+j+a}{mi+a} u^i v^j = (1-v)^{2\varrho} \sum_{i,j\ge 0} \binom{mi+j+a+2\varrho}{mi+a+2\varrho} u^i v^j.
	\end{align}
	To begin with, we have
	\begin{align*}
		\sum_{i,j\ge 0} \binom{mi+j+a}{mi+a} u^i v^j = \sum_{i\ge 0} u^i \sum_{j\ge 0} \binom{j+mi+a}{mi+a} v^j = \sum_{i\ge 0} \frac{u^i}{(1-v)^{mi+a+1}}, 
	\end{align*}
	where we have applied \eqref{eq:bin-gf}. Now \eqref{eq:rho-gf-1} becomes
	\begin{align*}
		\sum_{i\ge 0} \frac{u^i}{(1-v)^{mi+a+1}} = (1-v)^{2\varrho} \sum_{i\ge 0} \frac{u^i}{(1-v)^{mi+a+2\varrho+1}},
	\end{align*}
	which is obvious.
	
	Next, we show that
	\begin{align}\label{eq:rho-gf-2}
		\sum_{i,j\ge 0} \binom{mi-j+a}{mi+a} u^i v^j - \frac{1}{1-u} = (1-v)^{2\varrho} \sum_{i,j\ge 0} \binom{mi-j+a}{mi+a+2\varrho} u^i v^j.
	\end{align}
	Note that
	\begin{align*}
		&\sum_{i,j\ge 0} \binom{mi-j+a}{mi+a} u^i v^j\\
		&\qquad= \sum_{i\ge 0} u^i \sum_{j\ge 0} \binom{-j+mi+a}{mi+a} v^j\\
		&\qquad= \sum_{i\ge 0} u^i (-1)^{mi+a} \sum_{j\ge 0} \binom{j-1}{mi+a} v^j\\
		&\qquad= \sum_{i\ge 0} u^i (-1)^{mi+a} \binom{-1}{mi+a} - \sum_{i\ge 0} u^i (-v)^{mi+a+1} \sum_{j\ge 0} \binom{j+mi+a}{mi+a} v^{j}\\
		&\qquad= \frac{1}{1-u} - \sum_{i\ge 0} \frac{u^i (-v)^{mi+a+1}}{(1-v)^{mi+a+1}}.
	\end{align*}
	We similarly have, with the assumption $\varrho\ge 1$ in mind, that
	\begin{align*}
		\sum_{i,j\ge 0} \binom{mi-j+a}{mi+a+2\varrho} u^i v^j &= \sum_{i\ge 0} u^i (-1)^{mi+a+2\varrho} \sum_{j\ge 0} \binom{j+2\varrho-1}{mi+a+2\varrho} v^j\\
		&= - \sum_{i\ge 0} u^i (-v)^{mi+a+1} \sum_{j\ge 0} \binom{j+mi+a+2\varrho}{mi+a+2\varrho} v^{j}\\
		&=  - \sum_{i\ge 0} \frac{u^i (-v)^{mi+a+1}}{(1-v)^{mi+a+2\varrho+1}}.
	\end{align*}
	Combining the above two relations finally leads us to \eqref{eq:rho-gf-2}.
\end{proof}

For the H-Delannoy case, we have results of the same nature.

\begin{theorem}
	For $s\ge 1$ and $r\ge 0$,
	\begin{align}\label{eq:hL-110}
		\hcyL_{1,1,0}^{(s,r)}(n) = \tfrac{1}{2} \cyB_{0,0,c,c}^{(1,s+2)}(n) = (s+1)^{\binom{n}{2}}.
	\end{align}
	In addition, letting $\varrho\ge 1$ be a positive integer and assuming further that $r\ge 1$, then
	\begin{align}\label{eq:hL-101}
		\hcyL_{1,0,1}^{(s,r)}(n) = \cyB_{r+2\varrho-1,b,r+2\varrho-1,r-1}^{(s+1,1)}(n) = \tfrac{1}{2} \cyB_{r-1,b,r-1,r-1}^{(s+1,1)}(n).
	\end{align}
\end{theorem}

\begin{proof}
	It is clear from \eqref{eq:D-gf} and \eqref{eq:H-gf} that $H_{1,1,0}(i,j)=D_{1,1,0}(i,j)$; this can also be seen combinatorially from the facts that for each H-Delannoy path from $(-j,j)$ to $(s(n-j)+r-j,n+1)$, the position $(s(n-j)+r-j-1,n+1)$ is forbidden by definition and that all northeast steps are also not allowed since the weight $w_3=0$, so that this H-Delannoy path must pass through $(s(n-j)+r-j,n)$ and hence becomes a Delannoy path ending at $(s(n-j)+r-j,n)$ with the weights for the three types of steps being $(w_1,w_2,w_3)=(1,1,0)$ as assumed. Therefore, \eqref{eq:hL-110} comes directly from \eqref{eq:L-110}.
	
	For \eqref{eq:hL-101}, we derive from \eqref{eq:a+2delta} that for every $\varrho\ge 1$, the relation
	\begin{align*}
		\cyB_{r+2\varrho-1,b,r+2\varrho-1,r-1}^{(s+1,1)}(n) = \tfrac{1}{2} \cyB_{r-1,b,r-1,r-1}^{(s+1,1)}(n)
	\end{align*}
	holds when $r\ge 1$. Now it remains to take $\varrho = 1$ and $b=1$ and show
	\begin{align*}
		\hcyL_{1,0,1}^{(s,r)}(n) = \cyB_{r+1,1,r+1,r-1}^{(s+1,1)}(n),
	\end{align*}
	which has already been given in \eqref{eq:main-h}.
\end{proof}

\subsection*{Acknowledgements}

Shane Chern was supported by the Austrian Science Fund (No.~10.55776/F1002). SC and AY are grateful to Ilse Fischer for bringing \cite{KKS2025} to their attention, to Christian Krattenthaler for many useful comments on an earlier draft of the manuscript, and to Marcus Sch\"onfelder and Moritz Gangl for helpful discussions. QC is grateful to Amazon Web Services for the computational support.

\bibliographystyle{amsplain}

\end{document}

%% file: tikz/aztec4.tex

\draw[ultra thick] (0,7) rectangle (1,8);
\draw[ultra thick] (1,8) rectangle (2,9);
\draw[ultra thick] (2,9) rectangle (3,10);
\draw[ultra thick] (3,10) rectangle (4,11);
\draw[ultra thick] (0,6) rectangle (1,7);
\draw[ultra thick] (1,7) rectangle (2,8);
\draw[ultra thick] (2,8) rectangle (3,9);
\draw[ultra thick] (3,9) rectangle (4,10);
\draw[ultra thick] (4,10) rectangle (5,11);
\draw[ultra thick] (0,5) rectangle (1,6);
\draw[ultra thick] (1,6) rectangle (2,7);
\draw[ultra thick] (2,7) rectangle (3,8);
\draw[ultra thick] (3,8) rectangle (4,9);
\draw[ultra thick] (4,9) rectangle (5,10);
\draw[ultra thick] (0,4) rectangle (1,5);
\draw[ultra thick] (1,5) rectangle (2,6);
\draw[ultra thick] (2,6) rectangle (3,7);
\draw[ultra thick] (3,7) rectangle (4,8);
\draw[ultra thick] (4,8) rectangle (5,9);
\draw[ultra thick] (5,9) rectangle (6,10);
\draw[ultra thick] (0,3) rectangle (1,4);
\draw[ultra thick] (1,4) rectangle (2,5);
\draw[ultra thick] (2,5) rectangle (3,6);
\draw[ultra thick] (3,6) rectangle (4,7);
\draw[ultra thick] (4,7) rectangle (5,8);
\draw[ultra thick] (5,8) rectangle (6,9);
\draw[ultra thick] (0,2) rectangle (1,3);
\draw[ultra thick] (1,3) rectangle (2,4);
\draw[ultra thick] (2,4) rectangle (3,5);
\draw[ultra thick] (3,5) rectangle (4,6);
\draw[ultra thick] (4,6) rectangle (5,7);
\draw[ultra thick] (5,7) rectangle (6,8);
\draw[ultra thick] (6,8) rectangle (7,9);
\draw[ultra thick] (0,1) rectangle (1,2);
\draw[ultra thick] (1,2) rectangle (2,3);
\draw[ultra thick] (2,3) rectangle (3,4);
\draw[ultra thick] (3,4) rectangle (4,5);
\draw[ultra thick] (4,5) rectangle (5,6);
\draw[ultra thick] (5,6) rectangle (6,7);
\draw[ultra thick] (6,7) rectangle (7,8);
\draw[ultra thick] (0,0) rectangle (1,1);
\draw[ultra thick] (2,2) rectangle (3,3);
\draw[ultra thick] (4,4) rectangle (5,5);
\draw[ultra thick] (6,6) rectangle (7,7);
\draw[fill=white,draw=gray] (0,7) rectangle (1,8);
\draw[fill=white,draw=gray] (1,8) rectangle (2,9);
\draw[fill=white,draw=gray] (2,9) rectangle (3,10);
\draw[fill=white,draw=gray] (3,10) rectangle (4,11);
\draw[fill=white,draw=gray] (0,6) rectangle (1,7);
\draw[fill=white,draw=gray] (1,7) rectangle (2,8);
\draw[fill=white,draw=gray] (2,8) rectangle (3,9);
\draw[fill=white,draw=gray] (3,9) rectangle (4,10);
\draw[fill=white,draw=gray] (4,10) rectangle (5,11);
\draw[fill=white,draw=gray] (0,5) rectangle (1,6);
\draw[fill=white,draw=gray] (1,6) rectangle (2,7);
\draw[fill=white,draw=gray] (2,7) rectangle (3,8);
\draw[fill=white,draw=gray] (3,8) rectangle (4,9);
\draw[fill=white,draw=gray] (4,9) rectangle (5,10);
\draw[fill=white,draw=gray] (0,4) rectangle (1,5);
\draw[fill=white,draw=gray] (1,5) rectangle (2,6);
\draw[fill=white,draw=gray] (2,6) rectangle (3,7);
\draw[fill=white,draw=gray] (3,7) rectangle (4,8);
\draw[fill=white,draw=gray] (4,8) rectangle (5,9);
\draw[fill=white,draw=gray] (5,9) rectangle (6,10);
\draw[fill=white,draw=gray] (0,3) rectangle (1,4);
\draw[fill=white,draw=gray] (1,4) rectangle (2,5);
\draw[fill=white,draw=gray] (2,5) rectangle (3,6);
\draw[fill=white,draw=gray] (3,6) rectangle (4,7);
\draw[fill=white,draw=gray] (4,7) rectangle (5,8);
\draw[fill=white,draw=gray] (5,8) rectangle (6,9);
\draw[fill=white,draw=gray] (0,2) rectangle (1,3);
\draw[fill=white,draw=gray] (1,3) rectangle (2,4);
\draw[fill=white,draw=gray] (2,4) rectangle (3,5);
\draw[fill=white,draw=gray] (3,5) rectangle (4,6);
\draw[fill=white,draw=gray] (4,6) rectangle (5,7);
\draw[fill=white,draw=gray] (5,7) rectangle (6,8);
\draw[fill=white,draw=gray] (6,8) rectangle (7,9);
\draw[fill=white,draw=gray] (0,1) rectangle (1,2);
\draw[fill=white,draw=gray] (1,2) rectangle (2,3);
\draw[fill=white,draw=gray] (2,3) rectangle (3,4);
\draw[fill=white,draw=gray] (3,4) rectangle (4,5);
\draw[fill=white,draw=gray] (4,5) rectangle (5,6);
\draw[fill=white,draw=gray] (5,6) rectangle (6,7);
\draw[fill=white,draw=gray] (6,7) rectangle (7,8);
\draw[fill=white,draw=gray] (0,0) rectangle (1,1);
\draw[fill=white,draw=gray] (2,2) rectangle (3,3);
\draw[fill=white,draw=gray] (4,4) rectangle (5,5);
\draw[fill=white,draw=gray] (6,6) rectangle (7,7);

%% file: tikz/protodomain.tex

\draw[ultra thick] (0,1) rectangle (1,-8);
\draw[ultra thick] (1,2) rectangle (2,-7);
\draw[ultra thick] (2,3) rectangle (3,-6);
\draw[ultra thick] (3,4) rectangle (4,-5);
\draw[ultra thick] (4,4) rectangle (5,-4);
\draw[ultra thick] (5,3) rectangle (6,-3);
\draw[ultra thick] (6,2) rectangle (7,-2);
\draw[ultra thick] (7,1) rectangle (8,-1);
\foreach \i in {0,...,3}
{\draw[fill=white,draw=gray] (\i,\i) rectangle (\i+1,\i+1);}
\foreach \i in {0,...,4}
{\draw[fill=lightgray,draw=gray] (\i,\i-1) rectangle (\i+1,\i-1+1);}
\foreach \i in {0,...,4}
{\draw[fill=white,draw=gray] (\i,\i-2) rectangle (\i+1,\i-2+1);}
\foreach \i in {0,...,5}
{\draw[fill=lightgray,draw=gray] (\i,\i-3) rectangle (\i+1,\i-3+1);}
\foreach \i in {0,...,5}
{\draw[fill=white,draw=gray] (\i,\i-4) rectangle (\i+1,\i-4+1);}
\foreach \i in {0,...,6}
{\draw[fill=lightgray,draw=gray] (\i,\i-5) rectangle (\i+1,\i-5+1);}
\foreach \i in {0,...,6}
{\draw[fill=white,draw=gray] (\i,\i-6) rectangle (\i+1,\i-6+1);}
\foreach \i in {0,...,7}
{\draw[fill=lightgray,draw=gray] (\i,\i-7) rectangle (\i+1,\i-7+1);}
\foreach \i in {0,...,7}
{\draw[fill=white,draw=gray] (\i,\i-8) rectangle (\i+1,\i-8+1);}
\draw[color=cyan] (0,0) node[left,scale=.75]{diagonal 0} -- (4,4);
\draw[color=violet] (0,-1) node[left,scale=.75]{diagonal 1} -- (5,4);
\draw[color=cyan] (0,-2) node[left,scale=.75]{diagonal 2} -- (5,3);
\draw[color=violet] (0,-3) node[left,scale=.75]{diagonal 3} -- (6,3);
\draw[color=cyan] (0,-4) node[left,scale=.75]{diagonal 4} -- (6,2);
\draw[color=violet] (0,-5) node[left,scale=.75]{diagonal 5} -- (7,2);
\draw[color=cyan] (0,-6) node[left,scale=.75]{diagonal 6} -- (7,1);
\draw[color=violet] (0,-7) node[left,scale=.75]{diagonal 7} -- (8,1);
\draw[color=cyan] (0,-8) node[left,scale=.75]{diagonal 8} -- (8,0);

%% file: tikz/aztecdomain1.tex

\draw[fill=white,draw=white] (0,-8) rectangle (1,-7);
\foreach \i in {1, 4, 7, 10}
{\draw[fill=lightgray,draw=gray] (\i,\i-7) rectangle (\i+1,\i-7+1);}
\draw[line width=2.4pt,color=orange] (0,1) rectangle (1,-7);
\draw[line width=2.4pt,color=orange] (1,2) rectangle (2,-5);
\draw[line width=2.4pt,color=orange] (2,3) rectangle (3,-5);
\draw[line width=2.4pt,color=orange] (3,4) rectangle (4,-4);
\draw[line width=2.4pt,color=orange] (4,5) rectangle (5,-2);
\draw[line width=2.4pt,color=orange] (5,6) rectangle (6,-2);
\draw[line width=2.4pt,color=orange] (6,7) rectangle (7,-1);
\draw[line width=2.4pt,color=orange] (7,7) rectangle (8,1);
\draw[line width=2.4pt,color=orange] (8,6) rectangle (9,1);
\draw[line width=2.4pt,color=orange] (9,5) rectangle (10,2);
\foreach \i in {0,...,6}
{\draw[fill=white,draw=gray] (\i,\i) rectangle (\i+1,\i+1);}
\foreach \i in {0,...,7}
{\draw[fill=lightgray,draw=gray] (\i,\i-1) rectangle (\i+1,\i-1+1);}
\foreach \i in {0,...,7}
{\draw[fill=white,draw=gray] (\i,\i-2) rectangle (\i+1,\i-2+1);}
\foreach \i in {0,...,8}
{\draw[fill=lightgray,draw=gray] (\i,\i-3) rectangle (\i+1,\i-3+1);}
\foreach \i in {0,...,8}
{\draw[fill=white,draw=gray] (\i,\i-4) rectangle (\i+1,\i-4+1);}
\foreach \i in {0,...,9}
{\draw[fill=lightgray,draw=gray] (\i,\i-5) rectangle (\i+1,\i-5+1);}
\foreach \i in {0,...,9}
{\draw[fill=white,draw=gray] (\i,\i-6) rectangle (\i+1,\i-6+1);}
\foreach \i in {0, 2, 3, 5, 6, 8, 9}
{\draw[fill=lightgray,draw=gray] (\i,\i-7) rectangle (\i+1,\i-7+1);}
\foreach\i in {1, 4, 7, 10} {
	\draw[black,fill=black] (\i+.5,\i-7+.5) circle (.1);
}
\foreach\i in {0, 2, 3, 5, 6, 8, 9} {
	\draw[black,fill=white] (\i+.5,\i-7+.5) circle (.1);
}

%% file: tikz/aztecdomain2.tex

\foreach \i in {0, 2, 3, 5, 6, 8, 9}
{\draw[fill=white,draw=gray] (\i,\i-8) rectangle (\i+1,\i-8+1);}
\draw[line width=2.4pt,color=orange] (0,1) rectangle (1,-7);
\draw[line width=2.4pt,color=orange] (1,2) rectangle (2,-7);
\draw[line width=2.4pt,color=orange] (2,3) rectangle (3,-5);
\draw[line width=2.4pt,color=orange] (3,4) rectangle (4,-4);
\draw[line width=2.4pt,color=orange] (4,5) rectangle (5,-4);
\draw[line width=2.4pt,color=orange] (5,6) rectangle (6,-2);
\draw[line width=2.4pt,color=orange] (6,7) rectangle (7,-1);
\draw[line width=2.4pt,color=orange] (7,7) rectangle (8,-1);
\draw[line width=2.4pt,color=orange] (8,6) rectangle (9,1);
\draw[line width=2.4pt,color=orange] (9,5) rectangle (10,2);
\draw[line width=2.4pt,color=orange] (10,4) rectangle (11,2);
\foreach \i in {0,...,6}
{\draw[fill=white,draw=gray] (\i,\i) rectangle (\i+1,\i+1);}
\foreach \i in {0,...,7}
{\draw[fill=lightgray,draw=gray] (\i,\i-1) rectangle (\i+1,\i-1+1);}
\foreach \i in {0,...,7}
{\draw[fill=white,draw=gray] (\i,\i-2) rectangle (\i+1,\i-2+1);}
\foreach \i in {0,...,8}
{\draw[fill=lightgray,draw=gray] (\i,\i-3) rectangle (\i+1,\i-3+1);}
\foreach \i in {0,...,8}
{\draw[fill=white,draw=gray] (\i,\i-4) rectangle (\i+1,\i-4+1);}
\foreach \i in {0,...,9}
{\draw[fill=lightgray,draw=gray] (\i,\i-5) rectangle (\i+1,\i-5+1);}
\foreach \i in {0,...,9}
{\draw[fill=white,draw=gray] (\i,\i-6) rectangle (\i+1,\i-6+1);}
\foreach \i in {0,...,10}
{\draw[fill=lightgray,draw=gray] (\i,\i-7) rectangle (\i+1,\i-7+1);}
\foreach \i in {1, 4, 7, 10}
{\draw[fill=white,draw=gray] (\i,\i-8) rectangle (\i+1,\i-8+1);}
\foreach\i in {1, 4, 7, 10} {
	\draw[black,fill=black] (\i+.5,\i-8+.5) circle (.1);
}
\foreach\i in {0, 2, 3, 5, 6, 8, 9} {
	\draw[black,fill=white] (\i+.5,\i-8+.5) circle (.1);
}

%% file: tikz/bij1.tex

\draw[ultra thick] (0,1) rectangle (1,-7);
\draw[ultra thick] (1,2) rectangle (2,-5);
\draw[ultra thick] (2,3) rectangle (3,-5);
\draw[ultra thick] (3,4) rectangle (4,-4);
\draw[ultra thick] (4,5) rectangle (5,-2);
\draw[ultra thick] (5,6) rectangle (6,-2);
\draw[ultra thick] (6,7) rectangle (7,-1);
\draw[ultra thick] (7,7) rectangle (8,1);
\draw[ultra thick] (8,6) rectangle (9,1);
\draw[ultra thick] (9,5) rectangle (10,2);
\foreach \i in {0,...,6}
{\draw[fill=white,draw=gray] (\i,\i) rectangle (\i+1,\i+1);}
\foreach \i in {0,...,7}
{\draw[fill=lightgray,draw=gray] (\i,\i-1) rectangle (\i+1,\i-1+1);}
\foreach \i in {0,...,7}
{\draw[fill=white,draw=gray] (\i,\i-2) rectangle (\i+1,\i-2+1);}
\foreach \i in {0,...,8}
{\draw[fill=lightgray,draw=gray] (\i,\i-3) rectangle (\i+1,\i-3+1);}
\foreach \i in {0,...,8}
{\draw[fill=white,draw=gray] (\i,\i-4) rectangle (\i+1,\i-4+1);}
\foreach \i in {0,...,9}
{\draw[fill=lightgray,draw=gray] (\i,\i-5) rectangle (\i+1,\i-5+1);}
\foreach \i in {0,...,9}
{\draw[fill=white,draw=gray] (\i,\i-6) rectangle (\i+1,\i-6+1);}
\foreach \i in {0, 2, 3, 5, 6, 8, 9}
{\draw[fill=lightgray,draw=gray] (\i,\i-7) rectangle (\i+1,\i-7+1);}
\foreach\i\j in {0/-7, 0/-5, 0/-1, 1/0, 2/-3, 3/-4, 3/-2, 3/0, 4/1, 5/-2, 6/-1, 7/2}
{
	\draw[line width=2pt,draw=blue] (\i,\j) rectangle (\i+1,\j+2);
	\draw[line width=2pt,draw=red] (\i,\j+0.5) -- (\i+1,\j+1.5);
}
\foreach\i\j in {2/0, 4/-2}
{
	\draw[line width=2pt,draw=blue] (\i,\j) rectangle (\i+1,\j+2);
	\draw[line width=2pt,draw=red] (\i,\j+1.5) -- (\i+1,\j+0.5);
}
\foreach\i\j in {0/-3, 1/-4, 5/2, 8/3}
{
	\draw[line width=2pt,draw=blue] (\i,\j) rectangle (\i+2,\j+1);
	\draw[line width=2pt,draw=red] (\i,\j+0.5) -- (\i+2,\j+0.5);
}
\foreach\i\j in {0/-2, 1/-5, 1/-1, 2/2, 3/3, 4/0, 4/4, 5/1, 5/3, 5/5, 6/4, 6/6, 7/1, 7/5, 8/2, 8/4}
{
	\draw[line width=2pt,draw=blue] (\i,\j) rectangle (\i+2,\j+1);
}
\foreach\i\j in {0/-6.5, 0/-4.5, 0/-2.5, 0/-0.5, 1/-5.5, 4/-2.5, 7/0.5, 10/3.5}
{
	\filldraw[red] (\i,\j) circle[radius=4pt];
}

%% file: tikz/bij2.tex

\draw[ultra thick] (0,1) rectangle (1,-7);
\draw[ultra thick] (1,2) rectangle (2,-7);
\draw[ultra thick] (2,3) rectangle (3,-5);
\draw[ultra thick] (3,4) rectangle (4,-4);
\draw[ultra thick] (4,5) rectangle (5,-4);
\draw[ultra thick] (5,6) rectangle (6,-2);
\draw[ultra thick] (6,7) rectangle (7,-1);
\draw[ultra thick] (7,7) rectangle (8,-1);
\draw[ultra thick] (8,6) rectangle (9,1);
\draw[ultra thick] (9,5) rectangle (10,2);
\draw[ultra thick] (10,4) rectangle (11,2);
\foreach \i in {0,...,6}
{\draw[fill=white,draw=gray] (\i,\i) rectangle (\i+1,\i+1);}
\foreach \i in {0,...,7}
{\draw[fill=lightgray,draw=gray] (\i,\i-1) rectangle (\i+1,\i-1+1);}
\foreach \i in {0,...,7}
{\draw[fill=white,draw=gray] (\i,\i-2) rectangle (\i+1,\i-2+1);}
\foreach \i in {0,...,8}
{\draw[fill=lightgray,draw=gray] (\i,\i-3) rectangle (\i+1,\i-3+1);}
\foreach \i in {0,...,8}
{\draw[fill=white,draw=gray] (\i,\i-4) rectangle (\i+1,\i-4+1);}
\foreach \i in {0,...,9}
{\draw[fill=lightgray,draw=gray] (\i,\i-5) rectangle (\i+1,\i-5+1);}
\foreach \i in {0,...,9}
{\draw[fill=white,draw=gray] (\i,\i-6) rectangle (\i+1,\i-6+1);}
\foreach \i in {0,...,10}
{\draw[fill=lightgray,draw=gray] (\i,\i-7) rectangle (\i+1,\i-7+1);}
\foreach \i in {1, 4, 7, 10}
{\draw[fill=white,draw=gray] (\i,\i-8) rectangle (\i+1,\i-8+1);}
\foreach\i\j in {0/-7, 0/-5, 0/-1, 2/-3, 2/-1, 3/-2, 3/0, 4/-1, 8/1}
{
	\draw[line width=2pt,draw=blue] (\i,\j) rectangle (\i+1,\j+2);
	\draw[line width=2pt,draw=red] (\i,\j+0.5) -- (\i+1,\j+1.5);
}
\foreach\i\j in {1/-7, 1/-1, 7/-1}
{
	\draw[line width=2pt,draw=blue] (\i,\j) rectangle (\i+1,\j+2);
	\draw[line width=2pt,draw=red] (\i,\j+1.5) -- (\i+1,\j+0.5);
}
\foreach\i\j in {0/-3, 1/-4, 3/-4, 4/1, 5/0, 6/1, 9/2}
{
	\draw[line width=2pt,draw=blue] (\i,\j) rectangle (\i+2,\j+1);
	\draw[line width=2pt,draw=red] (\i,\j+0.5) -- (\i+2,\j+0.5);
}
\foreach\i\j in {0/-2, 1/-5, 1/1, 2/2, 3/-3, 3/3, 4/-2, 4/2, 4/4, 5/-1, 5/3, 5/5, 6/2, 6/4, 6/6, 7/3, 7/5, 8/4, 9/3}
{
	\draw[line width=2pt,draw=blue] (\i,\j) rectangle (\i+2,\j+1);
}
\foreach\i\j in {0/-6.5, 0/-4.5, 0/-2.5, 0/-0.5, 2/-6.5, 5/-3.5, 8/-0.5, 11/2.5}
{
	\filldraw[red] (\i,\j) circle[radius=4pt];
}

%% file: CCY.bbl
\begin{thebibliography}{99}
	
	\bibitem{Bax1982}
	R. J. Baxter, \textit{Exactly Solved Models in Statistical Mechanics}, Academic Press, Inc., London, 1982.
	
	\bibitem{CCY2025}
	Q. Chen, S. Chern, and A. Yoshida, Resources for ``Domino tilings, nonintersecting lattice paths and subclasses of Koutschan--Krattenthaler--Schlosser determinants''. Available at \url{https://shanechern.github.io/kks/kks.html}.
	
	\bibitem{Chy2000}
	F. Chyzak, An extension of Zeilberger's fast algorithm to general holonomic functions, \textit{Discrete Math.} \textbf{217} (2000), no. 1-3, 115--134.
	
	\bibitem{CHK2023}
	S. Corteel, F. Huang, and C. Krattenthaler, Domino tilings of generalized Aztec triangles, submitted. Available at arXiv:2305.01774.
	
	\bibitem{FSA2024}
	I. Fischer and F. Schreier-Aigner, $(-1)$-enumerations of arrowed Gelfand--Tsetlin patterns, \textit{European J. Combin.} \textbf{120} (2024), Paper No. 103979, 19 pp.
	
	\bibitem{DF2021}
	P. Di Francesco, Twenty vertex model and domino tilings of the Aztec triangle, \textit{Electron. J. Combin.} \textbf{28} (2021), no. 4, Paper No. 4.38, 50 pp.
	
	\bibitem{DFG2020}
	P. Di Francesco and E. Guitter, Twenty-vertex model with domain wall boundaries and domino tilings, \textit{Electron. J. Combin.} \textbf{27} (2020), no. 2, Paper No. 2.13, 63 pp.
	
	\bibitem{vzGG2013}
	J. von zur Gathen and J. Gerhard, \textit{Modern Computer Algebra. Third Edition}, Cambridge University Press, Cambridge, 2013.
	
	\bibitem{GV1985}
	I. Gessel and G. Viennot, Binomial determinants, paths, and hook length formulae, \textit{Adv. in Math.} \textbf{58} (1985), no. 3, 300--321.
	
	\bibitem{Kau2009}
	M. Kauers, Guessing handbook, in: \textit{RISC Report Series}, Report No. 09-07, Johannes Kepler University, Austria, 2009.
	
	\bibitem{Kau2013}
	M. Kauers, The holonomic toolkit, in: \textit{Computer Algebra in Quantum Field Theory}, 119--144, Springer, Vienna, 2013.
	
	\bibitem{Kel1974}
	S. B. Kelland, Twenty-vertex model on a triangular lattice, \textit{Austral. J. Phys.} \textbf{27} (1974), 813--829.
	
	\bibitem{Kou2009}
	C. Koutschan, Advanced applications of the holonomic systems approach, Thesis (Ph.D.)--Johannes Kepler University, 2009.
	
	\bibitem{Kou2010a}
	C. Koutschan, A fast approach to creative telescoping, \textit{Math. Comput. Sci.} \textbf{4} (2010), no. 2-3, 259--266.
	
	\bibitem{Kou2010b}
	C. Koutschan, HolonomicFunctions (User's Guide), in: \textit{RISC Report Series}, Report No. 10-01, Johannes Kepler University, Austria, 2010.
	
	\bibitem{KKS2025}
	C. Koutschan, C. Krattenthaler, and M. J. Schlosser, Determinant evaluations inspired by Di Francesco's determinant for twenty-vertex configurations, \textit{J. Symbolic Comput.} \textbf{127} (2025), Paper No. 102352, 34 pp.
	
	\bibitem{Lin1973}
	B. Lindstr\"om, On the vector representations of induced matroids, \textit{Bull. London Math. Soc.} \textbf{5} (1973), 85--90.
	
	\bibitem{Zei1991}
	D. Zeilberger, The method of creative telescoping, \textit{J. Symbolic Comput.} \textbf{11} (1991), no. 3, 195--204.
	
	\bibitem{Zei2007}
	D. Zeilberger, The holonomic Ansatz. II. Automatic discovery(!) and proof(!!) of holonomic determinant evaluations, \textit{Ann. Comb.} \textbf{11} (2007), no. 2, 241--247.
	
\end{thebibliography}
